\documentclass[10pt]{amsart}

\usepackage[margin=2.5cm,bottom=2.5cm,top=2.5cm]{geometry}
\usepackage{mathtools}
\mathtoolsset{showonlyrefs}
\usepackage[T1]{fontenc} \usepackage[utf8]{inputenc}

\usepackage{amsmath,amssymb,amsthm}
\usepackage[dvips]{graphicx}
\usepackage{color}
\usepackage{hyperref}
\usepackage{wasysym}
\newtheorem{lem}{Lemma}[section]
\newtheorem{prop}{Proposition}[section]

\newtheorem{thm}{Theorem}[section]

\theoremstyle{definition}

\theoremstyle{remark}

\theoremstyle{remark}
\newtheorem{remark}{Remark}[section]
\numberwithin{equation}{section}

\newcommand{\C}{{\mathbb C}}
\newcommand{\N}{{\mathbb N}}

\newcommand{\R}{{\mathbb R}}
\newcommand{\T}{{\mathbb T}}

\definecolor{blu}{rgb}{0,0,1}

\newcommand{\vertiii}[1]{{\left\vert\kern-0.25ex\left\vert\kern-0.25ex\left\vert #1
    \right\vert\kern-0.25ex\right\vert\kern-0.25ex\right\vert}}

\renewcommand\Re{\mathrm{Re}\,} \renewcommand\Im{\mathrm{Im}\,}

\begin{document}
\title{Growth of Sobolev Norms for $2d$ NLS with harmonic potential}

\author{Fabrice Planchon}
  \address{ Sorbonne Universit\'e, CNRS, IMJ-PRG F-75005 Paris, France}
\email{fabrice.planchon@sorbonne-universite.fr}
\author{Nikolay Tzvetkov}
\address{D\'epartement de Math\'ematiques, Universit\'e de Cergy-Pontoise, 2, 
avenue Adolphe Chauvin, 95302 Cergy-Pontoise  
Cedex, France and Institut Universitaire de France}%
\email{nikolay.tzvetkov@u-cergy.fr}%
\author{Nicola Visciglia}
\address{Dipartimento di Matematica, Universit\`a di Pisa,
Largo Bruno Pontecorvo 5, 56100 Pisa, Italy}%
\email{nicola.visciglia@unipi.it}
\thanks{ The first author was supported by ERC grant ANADEL no 757996, 
the second author by ANR grant ODA (ANR-18-CE40-0020-01), 
the third author by PRIN grant 2020XB3EFL. N.V. acknowledge the Gruppo Nazionale per l' Analisi Matematica, la Probabilit\`a e le loro Applicazioni (GNAMPA) of the Istituzione Nazionale di Alta Matematica (INDAM)} 
\date{\today}

\maketitle
 \par \noindent
 
\centerline {\em Dedicated to Professor Vladimir Georgiev for his 65's birthday}
 
\begin{abstract}
We prove polynomial upper bounds on the growth of solutions to  $2d$ cubic NLS where the Laplacian is confined by the harmonic potential. Due to better bilinear effects our bounds improve on those available for the $2d$ cubic NLS in the periodic setting: our growth rate for a Sobolev norm of order $s=2k$, $k\in \mathbb{N}$, is $t^{2(s-1)/3+\varepsilon}$.
In the appendix we provide an direct proof, based on integration by parts, of bilinear estimates associated with the harmonic oscillator.
\end{abstract}

\section{Introduction}

In recent years, growth of Sobolev norms for solutions to nonlinear dispersive equations generated a huge interest, in relation with weak turbulence phenomena. Concerning upper bounds, we quote the pioneering work of Bourgain \cite{b1} and its extension in a series of subsequent papers (\cite{CDKS},\cite{CKO}, \cite{D}, \cite{PTV}, \cite{So}, \cite{S}, \cite{Z} to quote only a few of them).
On the other end, growth of Sobolev norm cannot occur
in settings where the  dispersive effect is too strong. For instance consider
the translation invariant cubic defocusing NLS on $\R^2$. Then \cite{Dod} proved the long standing conjecture that nonlinear solutions scatter to free waves when time goes to infinity and hence
no growth phenomena is possible in such setting.
\\

We are interested in the growth of solutions to the following nonlinear Schr\"odinger equation:
\begin{equation}\label{harmosc}
\begin{cases}
i \partial_t u +A u\pm u|u|^2=0, \quad (t,x)\in \R \times \R^2\\
u(0,x)=\varphi(x)\in {\mathcal H}^s
\end{cases}
\end{equation}
where $x=(x_1, x_2)$, the operator $A$ is the usual Laplacian with an harmonic potential,
\begin{equation}\label{Agoti}A=-\Delta + |x|^2, \quad \hbox{ where }
\Delta=\partial_{x_1}^2 + \partial_{x_2}^2, \quad  |x|^2=x_1^2+x_2^2\end{equation}
and
$\|\varphi\|_{{\mathcal H}^s}=\|A^{s/2} \varphi\|_{L^2}$, where in general we use the notation $L^p=L^p(\R^2)$.
We shall also denote $L^p_{t,x}=L^p(\R\times \R^2)$ to emphasize the Lebesgue space of space-time dependent functions.\\

Let us first comment briefly about the local Cauchy theory associated with \eqref{harmosc}.
By combining preservation of regularity for the linear flow, $\|e^{itA} \varphi\|_{{\mathcal H}^s}=\|\varphi\|_{{\mathcal H}^s}$
and that ${\mathcal H}^s$ is an algebra for $s>1$, one proves existence of a local solution to \eqref{harmosc} by fixed point; its local time of existence depends
on the ${\mathcal H}^s$ norm of the initial datum. Moreover the solution map is Lipschitz continuous. In order to globalize our solution one can rely on the Brezis-Gallou\"et inequality (see \cite{BG})
provided that
\begin{equation}\label{basichampr}\sup_{t\in (-T_{min}(\varphi), T_{max}(\varphi))} \|u(t,x)\|_{\mathcal H^1}<\infty\end{equation}
where $(-T_{min}(\varphi), T_{max}(\varphi))$, with $T_{min}(\varphi), T_{max}(\varphi)>0$,  is the maximal time interval of existence of the solution associated with \eqref{harmosc}. In particular, assuming \eqref{basichampr}, $T_{max}(\varphi)=T_{min}(\varphi)=\infty$ and a double exponential bound holds:
\begin{equation}\label{doubleexp}\|u(t,x)\|_{\mathcal H^s}\leq C \exp (C\exp (C |t|)).\end{equation}
Solutions to \eqref{harmosc}
satisfy the conservation of the Hamiltonian
$$\frac 12 \|u(t,x)\|_{{\mathcal H}^1}^2 \pm  \frac 14 \|u(t,x)\|_{L^4}^4=const\,,$$
therefore, in the defocusing case, \eqref{basichampr} is automatically satisfied, while in the focusing case it is not granted for free.
Of course,  by using more sophisticated tools, e.g. Bourgain's spaces $X^{s,b}$ 
associated with $i\partial_t + A$, one can deal with initial data at lower regularity than ${\mathcal H}^{1+\varepsilon}$. These $X^{s,b}$ spaces
will play a key role in our analysis as they allow us to exploit a bilinear effect associated with the propagator
$e^{itA}$. They will be defined in Section \ref{cauchysection} where we also provide
more useful facts about Cauchy theory.

Our main goal is to improve \eqref{doubleexp} and prove polynomial upper bounds
for the quantity $\|u(t,x)\|_{{\mathcal H}^s}$ when $t\rightarrow \pm \infty$ with $s>1$. Along the rest of the paper the following equivalence of norms will be useful: for every $s\geq 0$ there exist $C>0$ such that
\begin{equation}\label{equivalence}
\frac 1{C} (\|D^s u\|_{L^2}^2 + \|\langle x\rangle ^s u\|_{L^2}^2)  \leq
\|\varphi\|_{{\mathcal H}^s}^2\leq  C (\|D^s u\|_{L^2}^2 + \|\langle x \rangle^s u\|_{L^2}^2)  
\end{equation}
where $D^s$ is the operator associated with the Fourier multiplier $|\xi|^s$ and 
$\langle x \rangle=\sqrt{1+x_1^2+x_2^2}$.
The proof of the equivalence \eqref{equivalence} is a special case of a more general result proved in \cite{BoTo,DG}.
In particular establishing growth upper bounds on ${\mathcal H}^s$ norm of the solution is equivalent to establish polynomial bounds
on the classical Sobolev norms $H^s$ and the corresponding moment of order $s$. We now state our main result.

\begin{thm}\label{main} 
Let  $\epsilon>0$ and $k\in \N$.  For every global solution $u$ to \eqref{harmosc} such that $u(t,x)\in {\mathcal C}(\R, {\mathcal H}^{2k})$ and \begin{equation}\label{basicham}\sup_{t\in \mathbb{R}} \|u(t,x)\|_{\mathcal H^1}<\infty\end{equation}
there exists a constant $C$ such that
$$
\|D^{2k} u(t, x)\|_{L^{2}} + \|\langle x\rangle ^{2k} u(t,x)\|_{L^2}
\leq  C \langle t\rangle^{\frac{2(2k-1)}3+\epsilon}.$$
\end{thm}

Our bound may be compared to the corresponding bound for solutions to
NLS on a generic compact $2-d$ manifold $M^2$ and more specifically on the torus $\T^2$. In fact at the best of our knowledge  
the best known upper bound available  on the growth of the classical Sobolev norm  $H^{2k}(\T^2)$ for solutions to cubic NLS on $\T^2$
is $(1+t)^{2k-1+\epsilon}$, as proved in \cite{Z}, \cite{PTV}. Notice also that in our case we control the 
growth of the moments as well (see also \cite{TV} for a different perspective on the moments).

Theorem \ref{main} may also be compared with \cite[Theorem 2]{CDKS}, where the same bound on the growth of Sobolev norm
was achieved for the translation invariant cubic NLS posed on $\R^2$, at a time where Dodson's definitive result was not available.
As already mentioned, unlike the situation considered in Theorem \ref{main}, where in general scattering
theory is not available, in the euclidean setting one can deduce uniform boundedness of
high order  Sobolev norms, at least in the defocusing situation. 
Nevertheless bounds provided in Theorem \ref{main} are still meaningful and non trivial
in the flat case either, if one considers solutions to the focusing NLS
such that the $H^1$ norm is uniformly bounded. In fact 
under this assumption it is not true in general that the solutions scatter to a free wave and hence
the uniform boundedness of Sobolev norms is not granted.

It would be very interesting to construct solutions to the defocusing \eqref{harmosc} such that the $H^k$ norms do not remain bounded in time for some $k>1$. Unfortunately such results are rare in the context of canonical dispersive models 
(with the notable exception of \cite{HPTV}). 

\section{$X^{s,b}$ framework and linear estimates}

We first define $X^{s,b}$ spaces associated with the harmonic oscillator in dimension two: the spectrum of the harmonic oscillator is 
given by the following set of integers $\{2n+2, n\in \N\}$.
For every $n\in \N$ we shall denote by $\Pi_n$ the orthogonal projector on the eigenspace associated with
the eigenvalue $2n+2$. Then the $X^{s,b}$ norm is given by the expression
$$\|u\|_{X^{s,b}}^2= \sum_{n\in \N} (2n+2)^s \big\|\langle \tau + 2n+2\rangle^{b}  {\mathcal F}_{t\rightarrow \tau} (\Pi_n u(t,x)) \big\|^2_{L^2_{\tau, x}}$$
where $u(t,x)$ is a function globally defined on space-time and ${\mathcal F}_{t\rightarrow \tau}$ denotes the Fourier transform with respect to the time variable.
Along with the $X^{s,b}$ spaces, which are defined for global space-time functions, we also introduce 
its localized version for every $T>0$. More precisely for functions $v(t,x)$ on the strip $(-T, T)\times \R^2$ we define:
$$\|v\|_{X^{s,b}_T}=\inf_{\substack{\tilde v\in X^{s,b}\\ v(t,x)=\tilde v(t,x)_{|(-T,T)\times \R^2}}} \|\tilde v\|_{X^{s,b}}.$$
The main result of this section is the continuity  of suitable linear operators in the Bourgain's spaces
$X^{s,b}_T$. 
\begin{prop}\label{derivative}
For every $\delta\in (0, \frac 12)$, $b\in (0,1)$ there exists $C>0$ such that we have the following estimate for every $T>0$:
\begin{equation}\label{secondoprime}
\|L u\|_{X^{-\frac 12+\delta, \frac 12-\delta+2\delta b}_T}\leq C \|u\|_{X^{\frac 12+\delta,\frac 12-\delta+2\delta b}_T}
\end{equation}
\begin{equation}\label{secondosecond}
\|L u\|_{X^{\delta, (1-\delta)b}_T}\leq C \|u\|_{X^{1+\delta,(1-\delta)b}_T}
\end{equation}
where $L$ can be either $\partial_{x_i}$, $i=1,2$ or multiplication by $\langle x \rangle$.
\end{prop}
\begin{proof} 
We  prove Proposition \ref{derivative} without the time localization. The corresponding version in localized Bourgain' spaces is straightforward. We will prove the following bounds:
\begin{align}\label{primotir}\|L u\|_{X^{0, b}}\leq C \|u\|_{X^{1,b}}, \quad b\in [0,1]\\
\label{xsb213tir}\|L u\|_{X^{1, 0}}\leq C \|u\|_{X^{2,0}}.
\end{align}
Notice that \eqref{secondosecond} follows by interpolation between \eqref{primotir} and \eqref{xsb213tir}.
Moreover we get \begin{equation}\label{quartotir}\|L u\|_{X^{-1, 0}}\leq C \|u\|_{X^{0,0}}
\end{equation} 
by duality from 
\eqref{primotir} for $b=0$, and we also get \begin{equation}\label{quinto}\|L u\|_{X^{-1/2, 1/2}}\leq C\|u\|_{X^{1/2,1/2}}
\end{equation} 
by interpolation between \eqref{quartotir} and \eqref{primotir} for 
$b=1$. Then \eqref{secondoprime} follows, interpolating \eqref{primotir} and
\eqref{quinto}. Hence we focus on \eqref{primotir} and \eqref{xsb213tir}.
Since the proof is slightly different depending from the operator
$L$ that we consider, we consider two cases.\\
\\
{\em First case: proof of \eqref{primotir} and \eqref{xsb213tir} for $L u=\partial_{x_i} u$}
\\
\\ 
First we prove that, for space-time dependent functions $u(t,x)$ we have
\begin{equation}\label{xsb21tir}\|\partial_{x_i} u\|_{X^{0, 0}}\leq C \|u\|_{X^{1,0}}\,.
\end{equation}
This estimate is a consequence of the following one
for time independent functions $v(x)$:
$$\|\partial_{x_i} v \|_{L^2}\leq C \|\sqrt A v\|_{L^2}$$
that in turn follows by
$\|\sqrt A v\|_{L^2}= \|v\|_{{\mathcal H}^1}$ and by recalling \eqref{equivalence}
for $s=1$.
Next we prove
\begin{equation}
  \label{xsb2tir}\|\partial_{x_i} u\|_{X^{0,1}}\leq C \|u\|_{X^{1,1}}\,,
\end{equation}
and  by interpolation with \eqref{xsb21tir}, \eqref{primotir} will follow for $L=\partial_{x_i}$.
As
$\|w(t,x)\|_{X^{0,1}}$ is equivalent to  $\|(i\partial_t + A) w\|_{L^2_{t,x}}+\|w\|_{L^2_{t,x}}$, in order to get \eqref{xsb2tir} we estimate
\begin{align}\label{curcguid}\|(i\partial_t + A) \partial_{x_i} u\|_{L^2_{t,x}}+\|\partial_{x_i} u\|_{L^2_{t,x}}
 & = \|\partial_{x_i} (i\partial_t + A) u + [|x|^2, \partial_{x_i}] u\|_{L^2_{t,x}}+\|\partial_{x_i} u\|_{L^2_{t,x}}\\
 & \leq  \|\partial_{x_i} (i\partial_t + A) u\|_{L^2_{t,x}}
   + 2\||x| u\|_{L^2_{t,x}}+\|\partial_{x_i} u\|_{L^2_{t,x}}\,.
   \end{align}
By combining \eqref{xsb21tir} with the following identity
\begin{equation}\label{calcfunctir}
\|\sqrt A v\|_{L^2}^2= (Av, v)=\|\nabla_x v\|_{L^2}^2 + \||x|v\|_{L^2}^2\end{equation}
we can continue \eqref{curcguid} as follows:
\[
  (\dots) \leq \|(i\partial_t + A) u\|_{X^{1,0}}+ 3 \|\sqrt A u\|_{L^2_{t,x}}
  \leq \|u\|_{X^{1,1}} + 3 \|u\|_{X^{1,0}}\leq 4 \|u\|_{X^{1,1}}\,.
  \]
and \eqref{xsb2tir}
for $L=\partial_{x_i}$ follows. Next we prove \eqref{xsb213tir}
(where $L=\partial_{x_i}$), namely
\begin{equation}\label{xsb2134tir}\|\partial_{x_i} u\|_{X^{1, 0}}\leq C \|u\|_{X^{2,0}}.
\end{equation}
This estimate is a consequence of the following one
for time independent functions $v(x)$:
$$\|\sqrt A\partial_{x_i} v \|_{L^2}\leq C \|A v\|_{L^2},$$
that in turn is equivalent to
\begin{equation}\label{calcfunc3}
(A \partial_{x_i} v, \partial_{x_i}v)\leq C (A v, A v).\end{equation}
As on the r.h.s. we get $\|v\|_{{\mathcal H}^2}^2$, by \eqref{equivalence} and elementary considerations it is sufficient to prove
\begin{equation}\label{....}
\int |x|^2 |\partial_{x_i}v|^2 \leq C( \|D^2 v\|_{L^2}^2 + \|\langle x \rangle^2 v\|_{L^2}^2)\,.
\end{equation}
In turn this last inequality follows by combining integration by parts and the Cauchy-Schwarz inequality:
\begin{align*}
  \int |x|^2 |\partial_{x_i}v|^2 =
-\int |x|^2 \partial_{x_i}^2v \bar v -2 \int  x_i \partial_{x_i}v \bar v 
 & \leq \| \partial_{x_i}^2v\|_{L^2} \||x|^2 v\|_{L^2}
+ 2 \||x|\partial_{x_i} v\|_{L^2} \|v\|_{L^2} 
\\
 & \leq \frac 12 \|D^2 u\|_{L^2}^2 + \frac 12  \|\langle x\rangle^2 v\|_{L^2}^2 + \frac 12 \||x|\partial_{x_i} v\|_{L^2}^2
+ 2  \|\langle x\rangle^2 v\|_{L^2}^2\,,
\end{align*}
from which we easily conclude moving $\frac 12 \||x|\partial_{x_i} v\|_{L^2}^2$ to the left-hand side.
\\
\\
{\em Second case: proof of \eqref{primotir} and \eqref{xsb213tir} for $L u=\langle x \rangle u$}
\\
\\
The proof follows the same steps as in the case $L=\partial_{x_i}$, with minor modifications.
First notice that  we have for space-time dependent functions $u(t,x)$ the following estimate:
\begin{equation}\label{xsb21tirtir}\|\langle x \rangle  u\|_{X^{0, 0}}\leq C \|u\|_{X^{1,0}}.
\end{equation}
This is a consequence of the following estimate
for time independent functions $v(x)$:
$$\|\langle x\rangle v \|_{L^2}\leq C \|\sqrt A v\|_{L^2}$$
that in turn follows by noticing that 
$\|\sqrt A v\|_{L^2}= \|v\|_{{\mathcal H}^1}$ and recalling \eqref{equivalence}
for $s=1$. Moreover we have
\begin{equation}\label{xsb2xtir}\|\langle x \rangle u\|_{X^{0,1}}\leq C \|u\|_{X^{1,1}}.
\end{equation}
that by interpolation with
\eqref{xsb21tirtir} implies \eqref{primotir}
for $L=\langle x \rangle$.
In order to prove this estimate recall again that 
$\|w(t,x)\|_{X^{0,1}}$ is equivalent to $\|i\partial_t w+ A w\|_{L^2_{t,x}}+\|w\|_{L^2_{t,x}}$
and hence we compute
\begin{align*}
\|(i\partial_t + A) (\langle x \rangle u)\|_{L^2_{t,x}}+\|\langle x \rangle u\|_{L^2_{t,x}}
 & = \|\langle x\rangle (i\partial_t + A) u + [\Delta, \langle x\rangle] u\|_{L^2_{t,x}}+\|\langle x \rangle u\|_{L^2_{t,x}}
\\
 & \leq  \|\langle x \rangle  (i\partial_t + A) u\|_{L^2_{t,x}}
+ \|2\nabla (\langle x \rangle) \cdot \nabla u + \Delta (\langle x \rangle) u \|_{L^2_{t,x}}
+\|\langle x \rangle u\|_{L^2_{t,x}}
\\
 & \leq C (\|\langle x \rangle (i\partial_t + A) u\|_{L^2_{t,x}} + \|\nabla u\|_{L^2_{t,x}}
+\|\langle x \rangle u\|_{L^2_{t,x}})\,.
\end{align*}
By combining \eqref{xsb21tirtir}  with the identity 
$\|\sqrt A u\|_{L^2_{t,x}}=\|u\|_{{\mathcal H}^1}$ and by recalling \eqref{equivalence}
for $s=1$,
we can proceed with our estimate above,
$$
(...)\leq C( \|(i\partial_t + A) u\|_{X^{1,0}}+ \|\sqrt A u\|_{L^2_{t,x}})
=C(\|u\|_{X^{1,1}} + \|u\|_{X^{1,0}}) \leq C \|u\|_{X^{1,1}}\,.
$$

Next we prove \eqref{xsb213tir}
(where $L=\langle x \rangle$), namely
\begin{equation}\label{xsb2134tir}\|\langle x \rangle u\|_{X^{1, 0}}\leq C \|u\|_{X^{2,0}}.
\end{equation}
This estimate is a consequence of the following one
for time independent functions $v(x)$:
$$\|\sqrt A (\langle x \rangle v) \|_{L^2}\leq C \|A v\|_{L^2}$$
that in turn is equivalent to
\begin{equation}\label{calcfunc3x}
(A (\langle x \rangle v), \langle x \rangle v)\leq C \|v\|_{{\mathcal H}^2}.\end{equation}
By \eqref{equivalence} it is equivalent to 
$$\|\nabla (\langle x \rangle v)\|_{L^2}^2 + \|\langle x\rangle |x| v\|_{L^2}^2
\leq C( \|D^2 v\|_{L^2}^2 + \|\langle x \rangle^2 v\|_{L^2}^2).$$
In turn, developing the gradient on the l.h.s. the estimate above follows from 
$$\int \langle x \rangle^2|\nabla v|^2\leq C( \|D^2 v\|_{L^2}^2 + \|\langle x \rangle^2 v\|_{L^2}^2)$$
whose proof proceeds by integration by parts and Cauchy-Schwarz inequality as we did for \eqref{....}.
\end{proof}
\section{The Cauchy theory in $X^{s,b}$ and consequences}\label{cauchysection}
We first obtain a trilinear estimate, whose
proof heavily relies on the analysis of \cite{Po} (also available as \cite{PoA}); for the sake of completeness, we provide a relatively elementary proof of the crucial bilinear estimate from \cite{Po} in the appendix, using the bilinear virial techniques from \cite{PlVe}. The only novelty in our trilinear estimate is that we prove a tame estimate, while such an estimate was not needed for the low regularity analysis of \cite{Po}.
We first recall the following key bilinear estimate 
\cite[Theorem~2.3.13]{Po}. There exists $\delta_0\in (0, \frac 12]$ such that for every $\delta\in (0,\delta_0]$ there exists $b'<\frac 12$ and $C>0$ such that:
 \begin{equation}\label{gg}
 \|\Delta_N (u)\, \Delta_M (v)\|_{L^2((0,T); L^2)}\leq C (\min(M,N))^{\delta}
 \Big(\frac{\min(M,N)}{\max(M,N)}\Big)^{\frac{1}{2}-\delta}
 \|\Delta_{N} (u)\|_{X^{0,b'}_T}
 \|\Delta_M (v)\|_{X^{0,b'}_T}
 \end{equation} 
where $\Delta_N$, $\Delta_M$ are the Littelwood-Paley localization associated with
$A$ and $N$, $M$ are dyadic integers.
\begin{prop}\label{poiret}
Let $0<T<1$ and $\epsilon>0$ be fixed. Then there exist $C>0$, $b>1/2$  and $\gamma>0$ such that for $s\geq \varepsilon$:
\begin{equation}\label{est}
\Big\|\int_{0}^t e^{i(t-\tau)A} (u_1(\tau)   u_2(\tau) \bar u_3(\tau)) d\tau\Big\|_{X^{s, b}_T}
\leq C T^\gamma \sum_{\sigma\in {\mathcal S}_3} \|u_{\sigma(1)}\|_{X^{s, b}_T} \|u_{\sigma(2)}\|_{X^{\varepsilon, b}_T}
\| u_{\sigma(3)}\|_{X^{\varepsilon, b}_T}.
\end{equation}
\end{prop} 
\begin{proof}
Using standard arguments (see for instance \cite[Proposition~3.3]{BGT_ens}), it suffices to prove that 
$$
\|u_1 u_2 \bar u_3 \|_{X^{s,-b'}_T}\leq C \sum_{\sigma\in {\mathcal S}_3} \|u_{\sigma(1)}\|_{X^{s, b}_T} \|u_{\sigma(2)}\|_{X^{\varepsilon, b}_T}\| u_{\sigma(3)}\|_{X^{\varepsilon, b}_T}
$$
for some $b>1/2$, $b'<1/2$ such that $b+b'<1$.
Using duality, the last estimate is equivalent to:
\begin{equation*}
\Big|
\int \int u_1 u_2 \bar u_3 \bar u_0
\Big|
\leq C
\|u_0\|_{X^{-s,b'}_T}
 \sum_{\sigma\in {\mathcal S}_3} \|u_{\sigma(1)}\|_{X^{s, b}_T} \|u_{\sigma(2)}\|_{X^{\varepsilon, b}_T}\| u_{\sigma(3)}\|_{X^{\varepsilon, b}_T}\,.
 \end{equation*}
 where $\int \int$ denotes a space-time integral on $\R^2\times \R$ with respect to the Lebesgue measure $dx dt$.
 We now perform a Littlewood-Paley decomposition in the left-hand side of the last inequality and  using a symmetry argument, we are reduced to obtaining a bound on
 $$
  \Big| \sum_{N_1\geq N_2 \geq N_3}\sum_{N_0}
\int \int  \Delta_{N_0}( \bar u_0) \Delta_{N_1} (u_1)  \Delta_{N_2} (u_2 ) \Delta_{N_3} (\bar u_3)\Big|,
 $$
 where the summation is meant over  dyadic values of $N_1$, $N_2$, $N_3$ and $N_0$. 
 The other possible orders of magnitudes of $N_1$,  $N_2$ and $N_3$ provide all permutations involved in the sum of the right hand-side of \eqref{est}.
 \\
 \\
 {\em First case: $N_0\geq N_1^{1+\delta}$ for some $\delta>0$}
 \\
 \\
 In this case, we can apply the $2d$ version of \cite[Lemme~2.1.23]{Po} to obtain that for every $K$ there is $C_K$ such that 
 \begin{equation*}
 \Big |\int \int  \Delta_{N_0}(\bar u_0) \Delta_{N_1} (u_1)  \Delta_{N_2} (u_2 ) \Delta_{N_3} (\bar u_3)\Big|
 \leq C_K N_0^{-K} 
 \|\Delta_{N_0} u_0\|_{X^{0,b'}_T} \|\Delta_{N_1} u_{1}\|_{X^{0, b'}_T} \|\Delta_{N_2} u_{2}\|_{X^{0, b'}_T}\|
 \Delta_{N_3} u_{3}\|_{X^{0, b'}_T}\,.
\end{equation*} 
where $b'<\frac 12$.
Now we can readily perform the $N_0$, $N_1$, $N_2$, $N_3$ summations thanks to the large negative power of $N_0$.\\
 \\
 {\em Second case: $N_0\leq N_1^{1+\delta}$, with $\delta>0$ to be chosen later depending  on $\varepsilon$}
 \\
 \\
Combining  Cauchy-Schwarz and \eqref{gg}, we  write
 \begin{align*}
 \Big |\int \int  \Delta_{N_0}(\bar u_0) \Delta_{N_1} (u_1)  \Delta_{N_2} (u_2 ) \Delta_{N_3} (\bar u_3) \Big|\
 \leq & \|\Delta_{N_1} (u_1)  \Delta_{N_2} (u_2 ) \|_{L^2((0,T);L^2)} \|\Delta_{N_0}( u_0) \Delta_{N_3} (u_3)  \|_{L^2((0,T); L^2)}
 \\
 \leq  & C
 (N_2 N_3)^{\delta}
  \Big(\frac{N_2 N_3}{N_0 N_1}\Big)^{\frac{1}{2}-\delta}
 \prod_{j=1}^4
 \|\Delta_{N_j} (u_j)\|_{X^{0,b'}_T}\,.
  \end{align*}
  A normalization yields that it suffices to prove the following inequality:
   \begin{multline*}
  \sum_{N_1\geq N_2\geq N_3} \sum_{N_0\leq N_1^{1+\delta}}
   (N_2 N_3)^{\delta}
  \Big(\frac{N_2 N_3}{N_0 N_1}\Big)^{\frac{1}{2}-\delta}
  N_{0}^{s}N_1^{-s}(N_2 N_3)^{-\varepsilon}\\
\!\!\!\!\!\!\!\!\!\!\!\!\!\!\!\!\!\!\!\!\!\!\!\!\!\!\!\!\!\!\!\!\!\!\!\!
    \times 
  \|\Delta_{N_0} (u_0)\|_{X^{-s,b'}_T} \|\Delta_{N_1} (u_1)\|_{X^{s,b'}_T}
  \prod_{j=3}^4 \|\Delta_{N_j} (u_j)\|_{X^{\epsilon,b'}_T}
  \\
  \leq C \Big (\sum_{N}  \|\Delta_{N} (u_0)\|_{X^{-s ,b'}_T}^2  \Big )^{1/2}
  \Big (\sum_{N}  \|\Delta_{N} (u_1)\|_{X^{s ,b'}_T}^2  \Big )^{1/2}
  \prod_{j=3}^4 \Big (\sum_{N}  \|\Delta_{N} (u_j)\|_{X^{\epsilon ,b'}_T}^2  \Big )^{1/2}\,.
   \end{multline*}
In the range of summation,
\begin{equation*}
 (N_2 N_3)^{\delta}
  \Big(\frac{N_2 N_3}{N_0 N_1}\Big)^{\frac{1}{2}-\delta}
  N_{0}^{s}N_1^{-s}(N_2 N_3)^{-\varepsilon}
 = N_{0}^{s-\frac 12+\delta}
  (N_2 N_3)^{\frac 12-\varepsilon} N_{1}^{-s-\frac 12+\delta}
\leq N_1^{-\kappa}
\end{equation*} where at the last step we have chosen $\delta>0$ small enough enough in such a way that
$\kappa>0$, allowing us to sum over  $N_0$, $N_1$, $N_2$, $N_3$.  This completes the proof of Proposition~\ref{poiret}. 
\end{proof}
As a standard consequence of Proposition~\ref{poiret} (see e.g. \cite[Proposition~3.3]{BGT_ens}), we can obtain the following well-posedness result.
\begin{prop}\label{Cauchy}
Let $R>0$ and $s_0\geq 1$ be given. Then there exists $T>0$ and $b>\frac 12$ such that 
\eqref{harmosc} has a unique local solution in $X^{s_0, b}_T$  for every $\varphi\in {\mathcal H}^{s_0}$
with $\|\varphi \|_{{\mathcal H}^{1}}<R$.
Moreover,
\begin{equation}\label{szero}
\|u(t,x)\|_{X^{s_0, b}_T}\leq 2 \|\varphi\|_{{\mathcal H}^{s_0}}\,.
\end{equation}
\end{prop}
\begin{remark}
While the Proposition is stated above regularity ${\mathcal H}^1$, it can be extended at lower regularity ${\mathcal H}^\varepsilon$
with $\varepsilon>0$, however we do not need such a low regularity later on.%
\end{remark}

Our next proposition reduces studying the growth of
the ${\mathcal H}^{2k}$ norm of the solution $u(t,x)$ to the analysis of 
the growth of $\|\partial_t^{k} u(t,x)\|_{L^2}$.
In fact this last quantity is easier to handle, as $\partial_{t}$ has better commutation properties with the nonlinear Schr\"odinger flow than the operator $A$.
\begin{prop}\label{equibalence}
Let $k,s\in \N$  and $R, \delta>0$ be given. Let $T>0$ be associated with $R$  and $s_0=2k+s$ as in Proposition \ref{Cauchy} and let
 $u(t,x)\in X^{2k+s, b}_T$ be the unique local solution to \eqref{harmosc} with initial condition $\varphi\in {\mathcal H}^{2k+s}$. 
Assume moreover that $\sup_{t\in (-T, T)} \|u(t,x)\|_{{\mathcal H}^{1}}<R$.
Then there exists $C>0$ such that
\begin{align}\label{eqENimproved2d}
\forall t\in (-T, T)\,,\quad\| \partial_t^k u(t)- i^k A^k u(t)\|_{{\mathcal H}^s}
\leq C\|u(t)\|_{{\mathcal H}^{s+2k-1}}^{1+\delta}\,.
\end{align}
\end{prop}

\begin{proof} We temporarily drop dependence on $t$ since the estimates we prove are pointwise in time.
In the sequel we shall also use without further comment that
$\sup_{t\in (-T,T)} \|u(t)\|_{{\mathcal H}^1}<R$. We shall denote by $\delta>0$ an arbitrary small number that can change from line to line. We start from the identity
\begin{equation}\label{basicequiv}
\partial_t^h u= i^h A^h u + \sum_{j=0}^{h-1} c_j \partial_t^j A^{h-j-1} (u|u|^2)\,,
\end{equation}
available for every integer $h\geq 1$ and for
suitable coefficients $c_j\in \C$. Its elementary
proof follows by induction on $h$, using the equation
solved by $u(t,x)$.

Next we argue by induction on $k$ in order to establish \eqref{eqENimproved2d}.
More precisely by assuming \eqref{eqENimproved2d} we shall prove that the same estimate is true if we replace $k$ by $k+1$.
Indeed by \eqref{basicequiv}, where we choose $h=k+1$, the estimate \eqref{eqENimproved2d} for $k+1$ reduces to
\begin{equation}\label{nosal}\|\partial_t^j (u|u|^{2})\|_{{\mathcal H}^{2k-2j+s}} \leq C
\|u\|_{{\mathcal H}^{s+2k+1}}^{1+\delta}, \quad j=0,\dots ,k.\end{equation}
Hence we prove \eqref{nosal},  assuming \eqref{eqENimproved2d}. Recalling \eqref{equivalence},
we have to prove
\begin{align}\label{nosal1}\|D^{2k-2j+s}\partial_t^j (u|u|^{2})\|_{L^2}\leq & C
\|u\|_{{\mathcal H}^{s+2k+1}}^{1+\delta} , \quad j=0,\dots ,k\,,\\
\label{nosal2}\|\langle x \rangle^{2k-2j+s} \partial_t^j (u|u|^{2})\|_{L^2}\leq  & C
\|u\|_{{\mathcal H}^{s+2k+1}}^{1+\delta}, \quad j=0,\dots ,k\,.\end{align}
To prove \eqref{nosal1}
we expand time and space derivatives on the left-hand side. Since $s$ is an integer and we never work with $L^1$ and $L^\infty$ norms, we may replace the operator $D$ 
by the usual gradient operator $\nabla$, and in particular, use the Leibniz rule. Hence by expanding space-time derivatives and by using H\"older,
we can estimate as follows the l.h.s. in \eqref{nosal1}:
\begin{align*}
\sum_{\substack{j_1+j_2+j_3=j\\s_1+s_2+s_3=2k-2j+s}} \prod_{l=1,2,3} \|\partial_t^{j_l} u\|_{W^{s_l, 6}}
& \leq  C \sum_{\substack{j_1+j_2+j_3=j\\s_1+s_2+s_3=2k-2j+s}}
\prod_{l=1,2,3} \|\partial_t^{j_l} u\|_{{\mathcal H}^{s_l+1}}\\\nonumber& \leq C
\sum_{\substack{j_1+j_2+j_3=j\\s_1+s_2+s_3=2k-2j+s}} \prod_{l=1,2,3} \|u
\|_{{\mathcal H}^{2j_l+s_l+1}}^{1+\delta}
\end{align*}
where we used the Sobolev embedding 
and the induction hypothesis at the last step.
We proceed with a trivial interpolation argument,
$$(\dots) \leq C  \|u\|_{{\mathcal H}^{s+2k+1}}^{\delta}
\Big( \prod_{l=1,2,3}  \|u\|_{{\mathcal H}^{s+2k+1}}^{\theta_l}\|u\|_{{\mathcal H}^{1}}^{(1-\theta_l)}
\Big),$$
where $\theta_l(s+2k+1)+ (1-\theta_l)= 2j_l+s_l+1$.
We conclude to \eqref{nosal1} since by direct computation we have $\sum_{l=1}^3 \theta_l=1$ for every $j=0, \dots ,k$.

We now turn to \eqref{nosal2}: by Leibniz rule and  H\"older, we estimate the l.h.s.
$$
\sum_{\substack{j_1+j_2+j_3=j\\ j_2, j_3<j}} \|\langle x \rangle^{2k-2j+s}\partial_t^{j_1} u\|_{L^{2}} \|\partial_t^{j_2} u\|_{L^\infty} 
\|\partial_t^{j_3} u\|_{L^\infty}
\leq C \sum_{\substack{j_1+j_2+j_3=j\\ j_2, j_3<j}} \|\partial_t^{j_1} u\|_{{\mathcal H}^{2k-2j+s}} 
\|\partial_t^{j_2} u\|_{{\mathcal H}^{1+\delta}} 
\|\partial_t^{j_3} u\|_{{\mathcal H}^{1+\delta}}
$$
where we used Sobolev embedding. By interpolation we proceed with
\begin{align*}
(\dots)  & \leq C \sum_{\substack{j_1+j_2+j_3=j\\ j_2, j_3<j}} \|\partial_t^{j_1} u\|_{{\mathcal H}^{2k-2j+s}} 
\|\partial_t^{j_2} u\|_{{\mathcal H}^1}^{1-\delta}  \|\partial_t^{j_2} u\|_{{\mathcal H}^2}^\delta  
\|\partial_t^{j_3} u\|_{{\mathcal H}^1}^{1-\delta}  \|\partial_t^{j_3} u\|_{{\mathcal H}^2}^\delta \\
 & \leq C \sum_{\substack{j_1+j_2+j_3=j\\ j_2, j_3<j}} \|u\|_{{\mathcal H}^{2j_1+2k-2j+s}}^{1+\delta} \|u\|_{{\mathcal H}^{2j_2+1}}
\|u\|_{{\mathcal H}^{2j_2+2}}^\delta
\|u\|_{{\mathcal H}^{2j_3+1}}\|u\|_{{\mathcal H}^{2j_3+2}}^\delta
\end{align*}
where we used the inductive assumption \eqref{eqENimproved2d}  to estimate $\|\partial_t^{j_l} u\|_{{\mathcal H}^s}$ for $l=2,3$ and $s=1,2$.
By a further interpolation step and using again \eqref{eqENimproved2d} we get
$$
(\dots) \leq C \|u\|_{{\mathcal H}^{2k+s}}^\delta
\Big( \prod_{l=1,2,3}  \|u\|_{{\mathcal H}^{s+2k+1}}^{\theta_l}\|u\|_{{\mathcal H}^{1}}^{(1-\theta_l)}
\Big)\,,
$$
where we have chosen
$$
\theta_1(s+2k+1)+ (1-\theta_1)= 2j_1+2k-2j+s\,, \,\,\,\,
\theta_l(s+2k+1)+ (1-\theta_l)= 2j_l+1, \quad l=2,3\,.
$$
We conclude to \eqref{nosal2} since one can check $\sum_{l=1}^3 \theta_l< 1$ for $j=0,\dots ,k$.
\end{proof}
The next proposition will be crucial in the sequel. It allows to estimate
the norm of time derivatives of the solution in the localized $X^{s,b}_T$ spaces,
by using suitable Sobolev norms of the initial datum.
\begin{prop}\label{cauchyderivative}

Let $l \in \N$, $R, \delta>0$ and $s\in (0,2]$ be given. Let $T>0$ be associated with $R$  and $s_0=2l+2$ as in Proposition \ref{Cauchy},
and $u(t,x)\in X^{2l+2, b}_T$ be the unique local solution to \eqref{harmosc} with initial condition $\varphi\in {\mathcal H}^{2l+2}$ and $\|\varphi \|_{{\mathcal H}^{1}}<R$.
Assume moreover that $\sup_{t\in (-T, T)} \|u(t,x)\|_{{\mathcal H}^{1}}<R$, then there exists $C>0$ such that:
\begin{equation}\label{012}
\|\partial_t^l u\|_{X^{s, b}_T}\leq C \|\varphi\|_{{\mathcal H}^{2l}}^{1-s}
\|\varphi\|_{{\mathcal H}^{2l+1}}^{s} \|\varphi\|_{{\mathcal H}^{2l+2}}^{\delta},\quad \hbox{ if }
s\in (0,1]\end{equation}
and
\begin{equation}\label{12}
\|\partial_t^l u\|_{X^{s, b}_T}\leq C \|\varphi\|_{{\mathcal H}^{2l+1}}^{2-s}
\|\varphi\|_{{\mathcal H}^{2l+2}}^{s-1+\delta} ,  \quad  \hbox{ if }
s\in (1,2].\end{equation}
\end{prop}

\begin{proof}
We shall prove separately \eqref{012} and \eqref{12} by induction on $l$. 
\\
\\
{\em Proof of \eqref{012}}
\\
\\
We consider the integral formulation of the equation solved by $\partial_t^l u$,
$$
\partial_t^l u(t)= e^{itA} \partial_t^l u(0) + \int_{0}^t e^{i(t-\tau)A} \partial_\tau^l (u(\tau)|u(\tau)|^2)
d\tau
$$
and then by standard properties of the $X^{s,b}$ spaces,
\begin{equation}\label{bouport}\|\partial_t^l u\|_{X^{s, b}_T}
\\\leq C\Big ( \|\partial_t^l u(0)\|_{{\mathcal H}^{s}} + \Big \|\int_{0}^t e^{i(t-\tau)A}
\partial_\tau^l (u(\tau)|u(\tau)|^2) d\tau\Big \|_{X^{s, b}_T}\Big )\,.
\end{equation}
Expanding the time derivative and using Proposition \ref{poiret} we get
\begin{equation}\label{bouport1new}
\|\partial_t^l u(t)\|_{X^{s, b}_T}
\leq C \Big (\|\partial_t^l u(0)\|_{{\mathcal H}^{s}} 
+T^\gamma \sum_{l_1+l_2+l_3=l} \|\partial_t^{l_1} u\|_{X^{s, b}_T}
 \|\partial_t^{l_2} u\|_{X^{s, b}_T}
\|\partial_t^{l_3} u \|_{X^{s, b}_T}\Big )\,.
\end{equation} 
By interpolation and Proposition \ref{equibalence} we also have 
\begin{equation}\label{datumcau}
\|\partial_t^l u(0)\|_{{\mathcal H}^{s}} \leq  \|\partial_t^l u(0)\|_{{\mathcal H}^{1}}^s\|\partial_t^l u(0)\|_{L^2}^{1-s}
\leq C \|u(0)\|_{{\mathcal H}^{2l+1}}^s \|u(0)\|_{{\mathcal H}^{2l}}^{1-s} 
\|u(0)\|_{{\mathcal H}^{2l+2}}^{\delta}= C\|\varphi\|_{{\mathcal H}^{2l+1}}^s \|\varphi\|_{{\mathcal H}^{2l}}^{1-s} 
\|\varphi\|_{{\mathcal H}^{2l+2}}^{\delta}\,.
\end{equation}
Therefore, estimating the second term on the r.h.s. in \eqref{bouport1new} is sufficient. We split the proof in two cases.\\
\\
{\em First case: $0<\min\{l_1, l_2, l_3\}\leq \max\{l_1, l_2, l_3\}<l$}
\\
\\
We use induction
on $l$ and estimate
\begin{multline*}\sum_{\substack{l_1+l_2+l_3=l\\
\max\{l_1, l_2, l_3\}<l}} \|\partial_t^{l_1} u\|_{X^{s, b}_T}
 \|\partial_t^{l_2} u\|_{X^{s, b}_T}
\|\partial_t^{l_3} u \|_{X^{s, b}_T}\\
\leq C\sum_{\substack{l_1+l_2+l_3=l\\
\max\{l_1, l_2, l_3\}<l}} \|\varphi\|_{{\mathcal H}^{2l_1}}^{1-s}
 \|\varphi\|_{{\mathcal H}^{2l_2}}^{1-s}
\|\varphi \|_{{\mathcal H}^{2l_3}}^{1-s}
\|\varphi\|_{{\mathcal H}^{2l_1+1}}^{s}
 \|\varphi\|_{{\mathcal H}^{2l_2+1}}^{s}
\|\varphi\|_{{\mathcal H}^{2l_3+1}}^{s} \|\varphi \|_{{\mathcal H}^{2l+2}}^{\delta} \\
\leq C
\|\varphi\|_{{\mathcal H}^{2l}}^{(1-s){\eta_1}}
\|\varphi\|_{{\mathcal H}^{2l}}^{(1-s)\eta_2}\|\varphi\|_{{\mathcal H}^{2l}}^{(1-s)\eta_3}
\|\varphi\|_{{\mathcal H}^{1}}^{(1-s)(1-\eta_1)}
\|\varphi\|_{{\mathcal H}^{1}}^{(1-s)(1-\eta_2)}\|\varphi\|_{{\mathcal H}^{1}}^{(1-s)(1-\eta_3)}\\
\times
\|\varphi\|_{{\mathcal H}^{2l+1}}^{s\theta_1}
\|\varphi\|_{{\mathcal H}^{2l+1}}^{s\theta_2}\|\varphi\|_{{\mathcal H}^{2l+1}}^{s\theta_3}
\|\varphi\|_{{\mathcal H}^{1}}^{s(1-\theta_1)}
\|\varphi\|_{{\mathcal H}^{1}}^{s(1-\theta_2)}\|u(t_0)\|_{{\mathcal H}^{1}}^{s(1-\theta_3)}
\|\varphi\|_{{\mathcal H}^{2l+2}}^{\delta} 
\end{multline*}
where
\begin{equation*}\begin{cases}
2l \eta_1 + 1-\eta_1= 2l_1\\
2l \eta_2 + 1-\eta_2= 2l_2\\
2l \eta_3 + 1-\eta_3= 2l_3
\end{cases}
 \,\,\,\,\,\text{and}\,\,\,\,\,\begin{cases}\theta_1(2l+1) + 1-\theta_1= 2l_1+1\\
\theta_2(2l+1) + 1-\theta_2= 2l_2+1\\
\theta_3(2l+1) + 1-\theta_3= 2l_3+1
\end{cases}\,.
\end{equation*}
Since $\theta_1+\theta_2+\theta_3=1$ and $\eta_1+\eta_2+\eta_3<1$ we may conclude with
\begin{equation}\label{l1l2l3diversi}\sum_{\substack{l_1+l_2+l_3=l\\
\max\{l_1, l_2, l_3\}<l}} \|\partial_t^{l_1} u\|_{X^{s, b}_T}
 \|\partial_t^{l_2} u\|_{X^{s, b}_T}
\|\partial_t^{l_3} u \|_{X^{s, b}_T}
\leq C \|\varphi\|_{{\mathcal H}^{2l}}^{1-s}
\|u(t_0)\|_{{\mathcal H}^{2l+1}}^{s} \|\varphi\|_{{\mathcal H}^{2l+2}}^{\delta}.
\end{equation}
\\
\\
{\em Second case: $0=\min\{l_1, l_2, l_3\}\leq \max\{l_1, l_2, l_3\}<l$}
\\
\\
We can assume $l_1=0$. Then we argue exactly as above except that, since
$\|\varphi\|_{{\mathcal H}^{2l_1}}= \|\varphi\|_{L^2}$ is bounded since we assume a control on the ${\mathcal H}^1$ norm of the initial datum, it is not necessary to introduce the parameter $\eta_1$.
Hence we need only $\eta_2, \eta_3, \theta_1, \theta_2, \theta_3$. The conclusion is the same as above.
\\
\\
{\em Third case: $\max\{l_1, l_2, l_3\}=l$}
\\
\\ 
We estimate the terms in the sum 
at the r.h.s. of \eqref{bouport1new} as follows
\begin{equation*}\label{lucben3}
\|u \|_{X^{s, b}_T}^2 \|\partial_t^l u \|_{X^{s, b}_T}\leq \|u \|_{X^{1, b}_T}^2 \|\partial_t^l u \|_{X^{s, b}_T}
\leq C
\|\varphi\|_{{\mathcal H}^1}^2\|\partial_t^{l} u\|_{X^{s, b}_T}
\end{equation*}
where we have used \eqref{szero} for $s_0=1$.
\\

Summarizing \eqref{bouport1new}, \eqref{datumcau} and \eqref{l1l2l3diversi} we get
\begin{equation*}
\|\partial_t^l u(t)\|_{X^{s, b}_T}
\leq C\Big (
\|\varphi\|_{{\mathcal H}^{2l+1}}^s \|\varphi\|_{{\mathcal H}^{2l}}^{1-s} 
\|\varphi\|_{{\mathcal H}^{2l+2}}^{\delta}
+T^\gamma \|\partial_t^l u\|_{X^{s, b}_T}\Big )\,.
\end{equation*} 
We conclude by choosing a time $\bar T>0$ small enough
in such a way that the second term on the r.h.s. can be absorbed by the l.h.s. Notice that the bound that we get on the short time $\bar T$ can be iterated since the constants depends only from the ${\mathcal H}^1$
norm of the solution and hence we get the desired bound up to the fixed time $T$ provided by the proposition after a finite iteration of the previous argument.
\\
\\
{\em Proof of \eqref{12}}
\\
\\
By Proposition \eqref{poiret} we get
\begin{equation}\label{pilad}
\|\partial_t^l u(t)\|_{X^{s, b}_T}
\leq C\Big( \|\partial_t^l u(0)\|_{{\mathcal H}^{s}}
+T^\gamma \sum_{l_1+l_2+l_3=l} \|\partial_t^{l_1} u\|_{X^{s, b}_T}
 \|\partial_t^{l_2} u\|_{X^{2-s, b}_T}
\|\partial_t^{l_3} u \|_{X^{2-s, b}_T}\Big ).
\end{equation} 
We first notice that by interpolation and Proposition \ref{equibalence} (see the proof of \eqref{datumcau}),
\begin{equation}\label{pilad1}
\|\partial_t^l u(0)\|_{{\mathcal H}^{s}}\leq C
\|\varphi\|_{{\mathcal H}^{2l+1}}^{2-s}
\|\varphi\|_{{\mathcal H}^{2l+2}}^{s-1+\delta}\,.
\end{equation}
Next we estimate the sum on the r.h.s. of \eqref{pilad} by considering
two cases.
\\
\\
{\em First case: $0<\min\{l_1, l_2, l_3\}\leq \max\{l_1, l_2, l_3\}<l$}
\\
\\
By combining the inductive assumption on $l$ 
and \eqref{012} we get: 
\begin{multline*}\sum_{\substack{l_1+l_2+l_3=l\\
\max\{l_1, l_2, l_3\}<l}} \|\partial_t^{l_1} u\|_{X^{s, b}_T}
 \|\partial_t^{l_2} u\|_{X^{2-s, b}_T}
\|\partial_t^{l_3} u \|_{X^{2-s, b}_T}\\
\leq C \sum_{\substack{l_1+l_2+l_3=l\\
\max\{l_1, l_2, l_3\}<l}} \|\varphi\|_{{\mathcal H}^{2l_1+1}}^{2-s}
 \|\varphi\|_{{\mathcal H}^{2l_2}}^{s-1}
\|\varphi\|_{{\mathcal H}^{2l_3}}^{s-1}
\|\varphi\|_{{\mathcal H}^{2l_1+2}}^{s-1}
 \|\varphi\|_{{\mathcal H}^{2l_2+1}}^{2-s}
\|\varphi \|_{{\mathcal H}^{2l_3+1}}^{2-s} \|\varphi \|_{{\mathcal H}^{2l+2}}^{\delta} \\
\leq C
\|\varphi\|_{{\mathcal H}^{2l+1}}^{(2-s){\eta_1}}
\|\varphi\|_{{\mathcal H}^{2l+2}}^{(s-1)\eta_2}\|\varphi\|_{{\mathcal H}^{2l+2}}^{(s-1)\eta_3}
\|\varphi\|_{{\mathcal H}^{1}}^{(2-s)(1-\eta_1)}
\|\varphi\|_{{\mathcal H}^{1}}^{(s-1)(1-\eta_2)}\|\varphi\|_{{\mathcal H}^{1}}^{(s-1)(1-\eta_3)}\\
\times
\|\varphi\|_{{\mathcal H}^{2l+2}}^{(s-1)\theta_1}
\|\varphi\|_{{\mathcal H}^{2l+1}}^{(2-s)\theta_2}\|\varphi\|_{{\mathcal H}^{2l+1}}^{(2-s)\theta_3}
\|\varphi\|_{{\mathcal H}^{1}}^{(s-1)(1-\theta_1)}
\|\varphi\|_{{\mathcal H}^{1}}^{(2-s)(1-\theta_2)}\|\varphi\|_{{\mathcal H}^{1}}^{(2-s)(1-\theta_3)}
\|\varphi\|_{{\mathcal H}^{2l+2}}^{\delta} 
\end{multline*}
where:
\begin{equation*}\begin{cases}\eta_1(2l+1) + 1-\eta_1= 2l_1+1\\
\eta_2(2l+2) + 1-\eta_2= 2l_2\\
\eta_3(2l+2) + 1-\eta_3= 2l_3
\end{cases}
\,\,\,\,\,\text{and}\,\,\,\,\,\,
\begin{cases}\theta_1(2l+2) + 1-\theta_1= 2l_1+2\\
\theta_2(2l+1) + 1-\theta_2= 2l_2+1\\
\theta_3(2l+1) + 1-\theta_3= 2l_3+1
\end{cases}.
\end{equation*}
By direct computation $\eta_1+\theta_2+\theta_3=1$ and $\theta_1+\eta_2+\eta_3<1$. Hence the estimate above implies
\begin{equation}\label{pilad2}\sum_{\substack{l_1+l_2+l_3=l\\
\max\{l_1, l_2, l_3\}<l}} \|\partial_t^{l_1} u\|_{X^{s, b}_T}
 \|\partial_t^{l_2} u\|_{X^{2-s, b}_T}
\|\partial_t^{l_3} u \|_{X^{2-s, b}_T}
\leq C
\|\varphi\|_{{\mathcal H}^{2l+1}}^{2-s}
\|\varphi\|_{{\mathcal H}^{2l+2}}^{s-1+\delta}.
\end{equation}
{\em Second case: $0=\min\{l_1, l_2, l_3\}\leq \max\{l_1, l_2, l_3\}<l$}
\\
\\
If $l_1=0$, then our previous proof is valid since we have to deal with the norm
$\|\varphi\|_{{\mathcal H}^{2l_1+1}}$ and hence we have regularity ${\mathcal H}^1$ and the interpolation argument above can be applied.
However, in the cases $l_2=0$ or $l_3=0$ the proof needs to be slightly modified.
We can assume $l_2=0$ then in this case
$\|\varphi\|_{{\mathcal H}^{2l_2}}= \|\varphi\|_{L^2}$ is bounded since we assume a control on the ${\mathcal H}^1$ norm of the initial datum, hence it is not necessary to introduce the parameter $\eta_2$
in the interpolation step.
The conclusion is the same as above.
\\
\\
{\em Third case:
$\max\{l_1, l_2, l_3\}=l$}
\\
\\
We have to consider three cases
$(l_1, l_2, l_3)=(l,0,0)$, $(l_1, l_2, l_3)=(0,l,0)$ and $(l_1, l_2, l_3)=(0,0,l)$ (the last two cases are similar).
In the first case we have
\begin{equation}\label{lucben1}
\|\partial_t^{l} u\|_{X^{s, b}_T} \|u\|_{X^{2-s, b}_T}
\|u \|_{X^{2-s, b}_T} \leq C \|\partial_t^{l} u\|_{X^{s, b}_T}
\end{equation}
where we used the estimate 
\begin{equation}\label{energyestXsb}\|u\|_{X^{2-s, b}_T}\leq C\|\varphi\|_{{\mathcal H}^1},\end{equation}
which is a consequence of \eqref{szero} where we choose $s_0=1$. We conclude this case by choosing $T$ small enough,
exactly as we did along the proof of \eqref{012}.
Finally, for $(l_1, l_2, l_3)=(0,l,0)$ 
by Proposition \ref{Cauchy}, \eqref{energyestXsb} and \eqref{012},
\begin{multline}\label{lucben3}
\|u\|_{X^{s, b}_T} \|\partial_t^l u\|_{X^{2-s, b}_T}
\|u \|_{X^{2-s, b}_T}\\
\leq C \|\varphi\|_{{\mathcal H}^s}\|\partial_t^{l} u\|_{X^{2-s, b}_T}
\leq C \|\varphi\|_{{\mathcal H}^s} \|\varphi\|_{{\mathcal H}^{2l}}^{s-1}
\|\varphi\|_{{\mathcal H}^{2l+1}}^{2-s}  \|\varphi\|_{{\mathcal H}^{2l+2}}^{\delta} 
\\\leq C \|\varphi\|_{{\mathcal H}^{2l+2}}^{\frac{s-1}{2l+1}} \|\varphi\|_{{\mathcal H}^{2l+2}}^{\frac{(s-1)(2l-1)}{1+2l}}
\|\varphi\|_{{\mathcal H}^{2l+1}}^{2-s}  \|\varphi\|_{{\mathcal H}^{2l+2}}^{\delta}
\end{multline}
where we used interpolation and the a priori bound on the $\mathcal H^1$ norm of the initial datum.
This concludes the proof of \eqref{12}.
\end{proof}
\section{Modified energies and proof of Theorem \ref{main}}
The aim of this section is to introduce suitable energies and to measure how far they are
from being exact conservation laws. Those energies are the key tool in order to achieve the growth estimate
provided in Theorem \ref{main}. Along this section we denote by $\int$ the integral on $\R^2$ with respect to the Lebesgue measure $dx$, and
$\int \int$ the integral on $\R^2\times \R$ with respect to the Lebesgue measure $dxdt$.
\begin{prop}\label{modifen} Let $u(t,x)\in \mathcal C((-T,T); {\mathcal H}^{2k+2})$ 
be a local solution to \eqref{harmosc} with initial datum
$\varphi\in {\mathcal H}^{2k+2}$. Then
we have:
\begin{equation}\label{tobeint}\frac d{dt} \Big (\frac 12\|\partial_t^k A u(t,x)\|_{L^2}^2 + {\mathcal S}_{2k+2} (u(t,x))\Big )
={\mathcal R}_{2k+2}(u(t,x))\end{equation}
where ${\mathcal S}_{2k+2} (u(t, x))$ 
is a linear combination of terms of the following type:
\begin{equation*}\label{gradient1}
\int \partial_t^{k} L  u_0 \partial_t^{m_1} L u_1
\partial_t^{m_2} u_2 \partial_t^{m_3} u_3,\\
\quad m_1+m_2+m_3=k, \quad m_1<k.
\end{equation*}
and ${\mathcal R}_{2k+2}(u(t,x))$
is a linear combination of terms of the following type:
\begin{equation}\label{gradient2}
\int \partial_t^{k} L  u_0 \partial_t^{l_1} L u_1
\partial_t^{l_2} u_2 \partial_t^{l_3} u_3,
\quad l_1+l_2+l_3=k+1, \quad l_1\leq k.
\end{equation}
where in \eqref{gradient1} and \eqref{gradient2} we have $u_0, u_1, u_2, u_3\in \{u, \bar u \}$ and 
$L$ can be any of the following operators:
$$Lu=\partial_{x_i} u \hbox{ for }i=1,2, \quad Lu=\langle x \rangle u, \quad Lu=u.$$
\end{prop}
\begin{proof}
We have
$$i \partial_t (\partial_t^k\sqrt A u) + A (\partial_t^k \sqrt A u) \pm \partial_t^k
\sqrt A (u|u|^2)=0\,.$$
Next we multiply the equation above by $\partial_t^{k+1} \sqrt A \bar u$ and we take the real part,
$$\frac 12 \frac d{dt} (
\|\partial_t^k A u\|_{L^2}^2)= \mp \Re \int \partial_t^k \sqrt A ( u |u|^2) \partial_t^{k+1} \sqrt{A} \bar u.$$
By symmetry of the operator $\sqrt A$ we have
\begin{multline*}\Re \int \partial_t^k \sqrt A ( u |u|^2) \partial_t^{k+1} \sqrt{A} \bar u=\Re \int  
\partial_t^{k} (u |u|^2) \partial_t^{k+1} A \bar u
\\=-\Re \int  \partial_t^{k} (u |u|^2) \partial_t^{k+1} \Delta \bar u + \Re \int  
\partial_t^{k} (u |u|^2) \partial_t^{k+1} (|x|^2 \bar u)\end{multline*}
and we proceed by integration by parts
\begin{equation}\label{grasbon}(\dots) =
\sum_{i=1}^2 \Re \int  \partial_t^{k}\partial_{x_i} (u |u|^2) \partial_t^{k+1} \partial_{x_i} \bar u + \Re \int |x|^2 
\partial_t^{k} (u|u|^2) \partial_t^{k+1}\bar  u.
\end{equation}
Next notice that the first term on the r.h.s. in \eqref{grasbon} can be written as follows
\begin{multline*}\sum_{i=1}^2  \Re \int  \partial_t^{k}\partial_{x_i} (u |u|^2) \partial_t^{k+1} \partial_{x_i} \bar u 
\\= \sum_{i=1}^2  \Big( 2 \Re  \int   |u|^2 \partial_t^{k}\partial_{x_i} u \partial_t^{k+1} \partial_{x_i} \bar u
+ \Re \int u^2\partial_t^{k}\partial_{x_i} \bar u  \partial_t^{k+1} \partial_{x_i} \bar u \Big)\\
+  \sum_{\substack{l_1+l_2+l_3=k\\\\max\{l_1, l_2, l_3\}<k}} \Re \Big(  a_{l_1, l_2, l_3} \partial_t^{l_1} \partial_{x_i} u \partial_t^{l_2} u \partial_t^{l_3} \bar u  \partial_t^{k+1} \partial_{x_i} \bar u
+b_{l_1, l_2, l_3} \partial_t^{l_1} \partial_{x_i} u \partial_t^{l_2} u \partial_t^{l_3} \partial_{x_i} \bar u  \partial_t^{k+1} \partial_{x_i} \bar u\Big)
\end{multline*}
where $a_{l_1, l_2, l_3}, b_{l_1, l_2, l_3}$ are suitable real numbers. Rewriting 
\begin{multline*}
(\dots) = \frac d{dt} \sum_{i=1}^2  \Big( \int | \partial_t^k \partial_{x_i} u|^2 |u|^2 
+\frac 12 \Re \int (\partial_t^k \partial_{x_i} \bar u)^2 u^2\Big) 
\\
-\sum_{i=1}^2  \Big( \int | \partial_t^k \partial_{x_i} u|^2 \partial_t (|u|^2) 
+ \frac 12 \Re \int (\partial_t^k \partial_{x_i} \bar u)^2 \partial_t (u^2)\Big)\\
+  \sum_{\substack{l_1+l_2+l_3=k\\\\l_1<k}} \Re \Big(  a_{l_1, l_2, l_3} \partial_t^{l_1} \partial_{x_i} u \partial_t^{l_2} u \partial_t^{l_3} \bar u  \partial_t^{k+1} \partial_{x_i} \bar u
+b_{l_1, l_2, l_3} \partial_t^{l_1} \partial_{x_i} u \partial_t^{l_2} u \partial_t^{l_3} \partial_{x_i} \bar u  \partial_t^{k+1} \partial_{x_i} \bar u\Big)\,,
\end{multline*}
by elementary manipulations on the last two lines we get
\begin{multline*}
(\dots) = \frac d{dt} \sum_{i=1}^2  \Big( \int | \partial_t^k \partial_{x_i} u|^2 |u|^2 
+\frac 12 \Re \int (\partial_t^k \partial_{x_i} \bar u)^2 u^2\Big) 
\\
-\sum_{i=1}^2  \Big( \int | \partial_t^k \partial_{x_i} u|^2 \partial_t (|u|^2) 
+ \frac 12 \Re \int (\partial_t^k \partial_{x_i} \bar u)^2 \partial_t (u^2)\Big)\\
+ \frac d{dt} \sum_{\substack{l_1+l_2+l_3=k\\\\l_1<k}} \Re \Big(  a_{l_1, l_2, l_3} \partial_t^{l_1} \partial_{x_i} u \partial_t^{l_2} u \partial_t^{l_3} \bar u  \partial_t^{k} \partial_{x_i} \bar u
+b_{l_1, l_2, l_3} \partial_t^{l_1} u \partial_t^{l_2} u \partial_t^{l_3} \partial_{x_i} \bar u  \partial_t^{k} \partial_{x_i} \bar u\Big)\\
+  \sum_{\substack{l_1+l_2+l_3=k+1\\l_1\leq k}} \Re \Big(  \tilde a_{l_1, l_2, l_3} \partial_t^{l_1} \partial_{x_i} u \partial_t^{l_2} u \partial_t^{l_3} \bar u  \partial_t^{k} \partial_{x_i} \bar u
+\tilde b_{l_1, l_2, l_3} \partial_t^{l_1} u \partial_t^{l_2} u \partial_t^{l_3} \partial_{x_i} \bar u  \partial_t^{k} \partial_{x_i} \bar u\Big)\,.
\end{multline*}
This last expression is a sum of a linear combination of terms with structure \eqref{gradient1} with $Lu=\partial_{x_i} u$ plus 
a time derivative of a linear combination of terms of type \eqref{gradient2} with $Lu=\partial_{x_i} u$.

Notice that the second term on the r.h.s. in \eqref{grasbon} rewrites
\[
\Re \int  |x|^2 \partial_t^{k} (u |u|^2) \partial_t^{k+1} \bar u = \Re \int  \partial_t^{k} (\langle
x \rangle u |u|^2) \partial_t^{k+1} (\langle x\rangle \bar u) 
-\Re \int  \partial_t^{k} (u |u|^2) \partial_t^{k+1} \bar u 
\]
and arguing as above one checks, by developing first the derivative of order $k$ with respect to time, that this expression is a time derivative of terms of type \eqref{gradient1}, where
$Lu=\langle x \rangle u$ or $Lu=u$,
plus a linear combination of terms with structure \eqref{gradient2} where $Lu=\langle x \rangle u$ or $Lu=u$.
\end{proof}
Next we estimate the energy ${\mathcal R}_{2k+2}$ introduced in Proposition \ref{modifen}.

\begin{prop}\label{R2k+2}  Let $k \in \N$, $R>0$ be given and 
$u(t,x)\in X^{2k+2, b}_T$ be the unique local solution to \eqref{harmosc} with initial condition $\varphi\in {\mathcal H}^{2k+2}$ and $\|\varphi \|_{{\mathcal H}^{1}}<R$,
where $T>0$ is associated with $R$  and $s_0=2k+2$ as in Proposition \ref{Cauchy}.
Assume moreover that $\sup_{t\in (-T, T)} \|u(t,x)\|_{{\mathcal H}^{1}}<R$. Then for every $\delta>0$ there exists $C>0$ such that:
$$\big| \int_{0}^T {\mathcal R}_{2k+2}(u(\tau,x)) d\tau\big |
\leq C \|\varphi\|_{\mathcal{H}^{2k+2}}^{\frac{8k+1}{4k+2}+\delta}.$$
\end{prop}

\begin{proof}
We have to estimate integrals like \eqref{gradient2}. After a Littlewood-Paley
decomposition we are reduced to estimating
\begin{equation*}\sum_{N_0, N_1, N_2, N_3}\int \int
\Delta_{N_0}(\partial_t^{k} L u_0) \Delta_{N_1} (\partial_t^{l_1} L u_1) 
\Delta_{N_2} (\partial_t^{l_2} u_2) \Delta_{N_3} (\partial_t^{l_3}u_3).\end{equation*}
where
$$ l_1+l_2+l_3=k+1, \quad l_1\leq k.$$
Here we have used the compact notation $\int \int $ to denote the space-time integral on the strip $(-T, T)\times \R^2$ and $\Delta_{N}$ denotes the Littlewood-Paley localization associated with the operator 
$A$ at dyadic frequency $N$.
We split the sum in several pieces depending on the frequencies
$N_0, N_1, N_2, N_3$ and we shall make extensively use of the following 
bilinear estimate (see \cite[Proposition~2.3.15]{Po}). 
For every $\delta\in (0,\frac 12]$, $b>\frac 12$ there exists $C>0$
such that: 
\begin{equation}\label{bilinptv}\|(\Delta_N u)(\Delta_M v)\|_{L^2((0, T);L^2)}
\leq C  \left ( \frac{\min\{N, M\}}{\max \{N, M\}} \right )^{\frac 12 - \delta} 
\|\Delta_N u\|_{X^{0, b}_T} \|\Delta_M u\|_{X^{0,b}_T}.
\end{equation}
Using the equation solved by $u$ and noticing that, with the imposed conditions on $l_1, l_2, l_3$ 
we may assume $l_2\geq 1$ (otherwise $l_3\geq 1$ and it is symmetric) we get
\begin{multline}\label{secon}\Big|\int \int
(\partial_t^k L u_0) (  \partial_t^{l_1} L u_1) \partial_t^{l_2} u_2  \partial_t^{l_3} u_3\Big|
\leq C \Big(\Big|\int \int
(\partial_t^k L u_0) (  \partial_t^{l_1} L u_1) (\partial_t^{l_2-1} Au_2)  \partial_t^{l_3} u_3 \Big|\\+  \int \int 
|(\partial_t^k L u_0)| |(  \partial_t^{l_1} L u_1)|| (\partial_t^{l_2-1} (u|u|^2)|  |\partial_t^{l_3} u|\Big)\,.
 \end{multline}
The second term on the right hand side is estimated  by Cauchy-Schwarz with 
\begin{multline*}
\Big(\int_0^T \|\partial_t^k L u_0(\tau)\|_{L^2} \|\partial_t^{l_1} L u_1\|_{L^2}d\tau\Big)
\| (\partial_t^{l_2-1} (u|u|^2)\|_{L^\infty((0,T);L^\infty)} \| \partial_t^{l_3} u\|_{L^\infty((0,T);L^\infty)} \\
\leq \Big(\int_0^T \|\partial_t^k L u_0(\tau)\|_{L^2} \|\partial_t^{l_1} L u_1\|_{L^2}d\tau\Big)
\| (\partial_t^{l_2-1} (u|u|^2)\|_{L^\infty((0,T);{\mathcal H}^{1+\delta})} \| \partial_t^{l_3} u\|_{L^\infty((0,T);{\mathcal H}^{1+\delta})} 
\end{multline*}
where we  used Sobolev embedding. Using \eqref{012} we proceed with
\begin{multline*}
(\dots) \leq C 
\|\varphi\|_{{\mathcal H}^{2k+1}} \|\varphi\|_{{\mathcal H}^{2l_1+1}} \|\varphi\|_{{\mathcal H}^{2l_3+1}} \|\varphi\|_{{\mathcal H}^{2k+2}}^{\delta} \| (\partial_t^{l_2-1} (u|u|^2)\|_{L^\infty((0,T);{\mathcal H}^{1+\delta})}\\\leq C 
\|\varphi\|_{{\mathcal H}^{2k+2}}^{\eta+\eta_1+\eta_3}\|\varphi\|_{{\mathcal H}^{1}}^{3-(\eta+\eta_1+\eta_3)} \|\varphi\|_{{\mathcal H}^{2k+2}}^{\delta} \| (\partial_t^{l_2-1} (u|u|^2)\|_{L^\infty((0,T);{\mathcal H}^{1+\delta})}
\end{multline*} 
where
$$
\begin{cases}
\eta(2k+2)+(1-\eta)=2k+1\\
\eta_i(2k+2)+(1-\eta_i)=2l_i+1, \quad i=1,3
\end{cases}
$$
and hence, by using the bound assumed on $\|\varphi\|_{{\mathcal H}^{1}}$ we can continue the estimate above as follows
$$\dots\leq C \|\varphi\|_{{\mathcal H}^{2k+2}}^{\frac{2k+2l_1+2l_3}{2k+1}}\|\varphi\|_{{\mathcal H}^{2k+2}}^{\delta} \| (\partial_t^{l_2-1} (u|u|^2)\|_{L^\infty((0,T);{\mathcal H}^{1+\delta})}\,.$$
Expanding $\partial_t^{l_2-1}(u|u|^2)$ and using that ${\mathcal H}^{1+\delta}$ is an algebra we get
\begin{align*}\| (\partial_t^{l_2-1} (u|u|^2)\|_{L^\infty((0,T);{\mathcal H}^{1+\delta})} \leq & C \sum_{j_1+j_2+j_3=l_2-1} \|\partial_t^{j_1} u\|_{L^\infty((0,T);{\mathcal H}^{1+\delta})}\|\partial_t^{j_2} u\|_{L^\infty((0,T);{\mathcal H}^{1+\delta})}\|\partial_t^{j_3} u\|_{L^\infty((0,T);{\mathcal H}^{1+\delta})}\\
\leq & C \sum_{j_1+j_2+j_3=l_2-1} \|\varphi\|_{{\mathcal H}^{2j_1+1}}\|u\|_{{\mathcal H}^{2j_2+1}}\|\varphi\|_{{\mathcal H}^{2j_3+1}} \|\varphi\|_{{\mathcal H}^{2k+2}}^{\delta} 
  \\\leq &C \sum_{j_1+j_2+j_3=l_2-1} \|\varphi\|_{{\mathcal H}^{2k+2}}^{\theta_1+\theta_2+\theta_3}\|\varphi\|_{{\mathcal H}^{1}}^{3-(\theta_1+\theta_2+\theta_3)}\|\varphi\|_{{\mathcal H}^{2k+2}}^{\delta}
  \end{align*}
where $$\theta_i(2k+2)+(1-\theta_i)=2j_i+1, \quad i=1,2,3$$ and recalling the a priori bound assumed on $\|\varphi\|_{{\mathcal H}^{1}}$
we conclude that
$$\| (\partial_t^{l_2-1} (u|u|^2)\|_{L^\infty((0,T);{\mathcal H}^{1+\delta})}\leq C  \|\varphi\|_{{\mathcal H}^{2k+2}}^{\frac{2l_2-2}{2k+1}}\,.$$
By combining the estimates above we get that the second term on the right-hand side  in \eqref{secon} 
can be estimated up to a constant by
$$\|\varphi\|_{{\mathcal H}^{2k+2}}^{\frac{2k+2l_1+2l_2+2l_3-2}{2k+1}+\delta}=\|\varphi\|_{{\mathcal H}^{2k+2}}^{\frac{4k}{2k+1}+\delta}.$$
We now focus on the first term on the right-hand side in \eqref{secon}
and by Littlewood-Paley
decomposition we are reduced to estimating
\begin{equation*}\sum_{N_0, N_1, N_2, N_3}\int \int
\Delta_{N_0}(\partial_t^{k} L u_0) \Delta_{N_1} (\partial_t^{l_1} L u_1) 
\Delta_{N_2} (\partial_t^{l_2-1} Au_2) \Delta_{N_3} (\partial_t^{l_3}u_3).\end{equation*}
Here we used again $\int \int $ to denote the space-time integral on the strip $(-T, T)\times \R^2$ and $\Delta_{N}$ still denotes the Littlewood-Paley localization associated with the operator 
$A$ at dyadic frequency $N$.
We split the sum in several pieces depending on the frequencies
$N_0, N_1, N_2, N_3$ and we shall make again extensive use of the bilinear estimate \eqref{bilinptv}. Next we consider several subcases.
\\
\\
{\em First subcase: $\min \{N_0, N_2\}\geq \max \{N_1, N_3\}$}
\\
\\
By Cauchy-Schwarz,
\begin{multline*}\Big |\int \int
\Delta_{N_0}(\partial_t^k L u_0) \Delta_{N_1} ( \partial_t^{l_1}L u_1) 
\Delta_{N_2} (\partial_t^{l_2-1}A u_2) \Delta_{N_3} (\partial_t^{l_3}u_3)\Big|
\\\leq \|
\Delta_{N_0}(\partial_t^k L u_0)
\Delta_{N_1} (\partial_t^{l_1}Lu_1)\|_{L^2((0, T);L^2)} \| \Delta_{N_2} (\partial_t^{l_2-1}A u) 
 \Delta_{N_3} (\partial_t^{l_3}u)\|_{L^2((0, T);L^2)}
\end{multline*}
and by \eqref{bilinptv} we can continue as follows
\begin{align*}
(\dots) \leq  & C \frac{(N_1 N_3)^{\frac 12-\delta}}{(N_0 N_2)^{\frac 12 -\delta}} 
\|\Delta_{N_0}(\partial_t^k  L u)\|_{X^{0,b}_T} \|\Delta_{N_1}(\partial_t^{l_1}L u)\|_{X^{0,b}_T}
\|\Delta_{N_2} (\partial_t^{l_2-1} A u_2)\|_{X^{0,b}_T}
\|\Delta_{N_3} (\partial_t^{l_3} u)\|_{X^{0,b}_T}\\
\leq & C \frac{N_3^{\frac 12-\delta}}{N_2^{\frac 12 -\delta}} 
\|\Delta_{N_0}( \partial_t^k L u)\|_{X^{0,b}_T} \|\Delta_{N_1}(\partial_t^{l_1} L u)\|_{X^{0,b}_T}
\|\Delta_{N_2} (\partial_t^{l_2-1}A u_2)\|_{X^{0,b}_T}
\|\Delta_{N_3} (\partial_t^{l_3}u)\|_{X^{0,b}_T}
\\
\leq & C
\|\Delta_{N_0}( \partial_t^k L u)\|_{X^{0,b}_T} 
\|\Delta_{N_1} ( \partial_t^{l_1}L u) \|_{X^{0,b}_T}
       \|\Delta_{N_2} ( \partial_t^{l_2-1}Au)\|_{X^{-\frac 12+\delta,b}_T}\|\Delta_{N_3} ( \partial_t^{l_3}u)\|_{X^{\frac 12-\delta,b}_T}\,.
       \end{align*}
Summarizing,
\begin{multline*}\sum_{\substack{N_0, N_1, N_2, N_3\\
\min\{N_0,N_2\}\geq  \max\{N_1, N_3\}}} 
\Big |\int \int
\Delta_{N_0}( L u^0) \Delta_{N_1} (L u_1) 
\Delta_{N_2} (A u_2) \Delta_{N_3} (u_3)\Big |
\\
\leq C \| L  \partial_t^k u\|_{X^{\delta,b}_T} 
\|L  \partial_t^{l_1}u \|_{X^{\delta,b}_T}
\|A \partial_t^{l_2-1}u\|_{X^{-\frac 12+\delta,b}_T} \| \partial_t^{l_3}u
\|_{X^{\frac 12-\delta,b}_T}
\\
\leq C \| \partial_t^k u\|_{X^{1+\delta,b}_T} 
\| \partial_t^{l_1} u\|_{X^{1+\delta,b}_T}
\|  \partial_t^{l_2-1}u\|_{X^{\frac 32+\delta,b}_T} \| \partial_t^{l_3}u
\|_{X^{\frac 12+\delta,b}_T}\end{multline*}
where we used Lemma \ref{derivative} at the last step,  assuming we chose $b>\frac 12$ in such a way the estimate at the last line fits with Lemma \ref{derivative}.
\\
\\
{\em Second subcase: $\min \{N_1, N_2\}\geq \max \{N_0, N_3\}$}
\\
\\
We can argue as above and we are reduced to the previous case by noticing that
we have the inequality 
$\frac{N_0 N_3}{N_1N_2}\leq \frac{N_3}{N_2}$ since in this subcase $N_0\leq N_1$.
\\
\\
{\em Third subcase: $\min \{N_3, N_2\}\geq \max \{N_0, N_1\}$}
\\
\\
We can argue as above and we are reduced to the first subcase by noticing that
we have the inequality 
$N_0N_1\leq N_3^2$  and hence $\frac{N_0N_1}{N_2N_3}\leq \frac{N_3}{N_2}$.
\\
\\
{\em Fourth subcase: $\min \{N_1, N_3\}\geq \max \{N_0, N_2\}$}
\\
\\
We can argue as above and we are reduced to the first subcase by noticing that
we have the inequality $\frac{N_0N_2}{N_1N_3}\leq \frac{N_3}{N_2}$. In fact it is equivalent to $N_0N_2^2\leq N_1N_3^2$ which clearly holds in that subcase.
\\
\\
{\em Fifth subcase: $\min \{N_0, N_3\}\geq \max \{N_1, N_2\}$}
\\
\\
We can argue as above and we are reduced to the first subcase by noticing that
we have the inequality $\frac{N_1N_2}{N_0N_3}\leq \frac{N_3}{N_2}$. In fact it is equivalent to $N_1N_2^2\leq N_0N_3^2$ which clearly holds in this subcase.
\\
\\
{\em Sixth subcase: $\min \{N_0, N_1\}\geq \max \{N_2, N_3\}$}
\\
\\
We can argue as above and we are reduced to the first subcase by noticing that
we have the inequality $\frac{N_2N_3}{N_0N_1}\leq \frac{N_3}{N_2}$ which in turn follows from the fact that in this subcase $N_2^2\leq N_0N_1$.
\\
\\
We are therefore left with proving
$$\| \partial_t^k u\|_{X^{1+\delta,b}_T} 
\| \partial_t^{l_1} u\|_{X^{1+\delta,b}_T}
\|  \partial_t^{l_2-1}u\|_{X^{\frac 32+\delta,b}_T} \| \partial_t^{l_3}u
\|_{X^{\frac 12+\delta,b}_T}\leq C  \|\varphi\|_{\mathcal{H}^{2k+2}}^{\frac{8k+1}{4k+2}+\delta}.$$
Using \eqref{012} and \eqref{12} we can control the right-hand side with
$$\|\varphi\|_{{\mathcal H}^{2k+1}}\|\varphi\|_{{\mathcal H}^{2l_1+1}}\|\varphi\|_{{\mathcal H}^{2l_2-1}}^\frac 12 
\|\varphi\|_{{\mathcal H}^{2l_2}}^\frac 12 \|\varphi\|_{{\mathcal H}^{2l_3}}^\frac 12 
\|\varphi\|_{{\mathcal H}^{2l_3+1}}^\frac 12 \|\varphi\|_{{\mathcal H}^{2k+2}}^\delta$$
which, by interpolation and recalling the a priori bound on ${\mathcal H}^1$ norm, can be estimated with 
$$\|\varphi\|_{{\mathcal H}^{2k+2}}^\theta \|\varphi\|_{{\mathcal H}^{2l_1+1}}^{\theta_1}\|\varphi\|_{{\mathcal H}^{2l_2-1}}^{\frac 12 \gamma_2} 
\|\varphi\|_{{\mathcal H}^{2l_2}}^{\frac 12 \eta_2} \|\varphi\|_{{\mathcal H}^{2l_3}}^{\frac 12 \eta_3} 
\|\varphi\|_{{\mathcal H}^{2l_3+1}}^{\frac 12 \theta_3} \|\varphi\|_{{\mathcal H}^{2k+2}}^\delta$$
where
$$\begin{cases}
\theta (2k+2)+(1-\theta)=2k+1\\
\theta_1 (2k+2)+(1-\theta_1)=2l_1+1\\
\theta_3 (2k+2)+(1-\theta_3)=2l_3+1\\
\gamma_2 (2k+2)+(1-\gamma_2)=2l_2-1\\
\eta_2 (2k+2)+(1-\eta_2)=2l_2\\
\eta_3 (2k+2)+(1-\eta_3)=2l_3.
\end{cases}
$$
We conclude by computing $\theta, \theta_1, \theta_3,\gamma_2, \eta_2, \eta_3$ and noticing that
$$\theta+\theta_1+\frac 12 \gamma_2+\frac 12 \eta_2+\frac 12 \eta_3+\frac 12 \theta_3=\frac{8k+1}{4k+2}\,,$$
and this concludes the proof.
 \end{proof}

Next we estimate the terms involved in the expression of ${\mathcal S}_{2k+2}$ introduced in Proposition \ref{modifen}.

\begin{prop}\label{S2k+2} 
Let $k \in \N$, $R>0$ be given and 
$u(t,x)\in X^{2k+2, b}_T$ be the unique local solution to \eqref{harmosc} with initial condition $\varphi\in {\mathcal H}^{2k+2}$ and $\|\varphi \|_{{\mathcal H}^{1}}<R$,
where $T>0$ is associated with $R$  and $s_0=2k+2$ as in Proposition \ref{Cauchy}.
Assume moreover that $\sup_{t\in (-T, T)} \|u(t,x)\|_{{\mathcal H}^{1}}<R$. Then for every $\delta>0$ there exists $C>0$ such that
$$\sup_{t\in (-T, T)}|{\mathcal S}_{2k+2}(u(t,x))|\leq C \|\varphi\|_{{\mathcal H}^{2k+2}}^{\frac{4k}{2k+1}+\delta}\,.$$
\end{prop}

\begin{proof}
We prove the desired estimate for every expression with type \eqref{gradient1}.
Indeed by H\"older we have for every fixed $t\in (-T, T)$
\[
    \int \partial_t^{k} L  u_0 \partial_t^{m_1} L u_1
\partial_t^{m_2} u_2 \partial_t^{m_3} u_3
\leq \|\partial_t^{k} L  u_0\|_{L^2} \|\partial_t^{m_1} L  u_1\|_{L^2} 
\|\partial_t^{m_2} u_2\|_{L^\infty} \| \partial_t^{m_3} u_3\|_{L^\infty}\,,
\]
and by Sobolev embedding and \eqref{equivalence} we proceed with
\begin{align*}
(\dots) \leq & C \|\partial_t^{k} L  u\|_{L^2} \|\partial_t^{m_1} L  u\|_{L^2}
\|\partial_t^{m_2} u\|_{{\mathcal H}^1}^{1-\delta} \|\partial_t^{m_2} u\|_{{\mathcal H}^2}^\delta  
\|\partial_t^{m_3} u\|_{{\mathcal H}^1}^{1-\delta} \|\partial_t^{m_3} u\|_{{\mathcal H}^2}^\delta
\\
\leq & C \|\partial_t^{k} u\|_{{\mathcal H}^1} \|\partial_t^{m_1} u\|_{{\mathcal H}^1}
\|\partial_t^{m_2} u\|_{{\mathcal H}^1}^{1-\delta} \|\partial_t^{m_2} u\|_{{\mathcal H}^2}^\delta  
\|\partial_t^{m_3} u\|_{{\mathcal H}^1}^{1-\delta} \|\partial_t^{m_3} u\|_{{\mathcal H}^2}^\delta.
\end{align*}
Then we can use \eqref{eqENimproved2d} to get
\begin{align*}
  (\dots) \leq & C \|u\|_{{\mathcal H}^{2k+1}} \|u\|_{{\mathcal H}^{2m_1+1}}
\|u\|_{{\mathcal H}^{2m_2+1}} \|u\|_{{\mathcal H}^{2m_3+1}} \|u\|_{{\mathcal H}^{2k+2}}^\delta\\
\leq & C \|\varphi\|_{{\mathcal H}^{2k+1}} \|\varphi \|_{{\mathcal H}^{2m_1+1}}
       \|\varphi\|_{{\mathcal H}^{2m_2+1}} \|\varphi\|_{{\mathcal H}^{2m_3+1}} \|\varphi\|_{{\mathcal H}^{2k+2}}^\delta
       \end{align*}
where $\delta>0$ is a small constant that can change from line to line and the last estimate follows from the embedding
$X^{s, b}_T\subset {\mathcal C}((-T, T); {\mathcal H}^s)$ for $b>\frac 12$ and \eqref{szero}.
Next we choose $\theta, \theta_1, \theta_2, \theta_3\in [0,1]$ such that
$$
\begin{cases}
\theta (2k+2)+ (1-\theta)=2k+1\\
\theta_i(2k+2)+ (1-\theta_i)= 2m_i+1, \quad i=1,2,3
\end{cases}
$$
and by interpolation,  we proceed with
\begin{equation*}\int \partial_t^{k} L  u_0 \partial_t^{m_1} L u_1
\partial_t^{m_2} u_2 \partial_t^{m_3} u_3 \leq C
 \|\varphi\|_{{\mathcal H}^{2k+2}}^{\theta+\theta_1+\theta_2+\theta_3} \|\varphi\|_{{\mathcal H}^{2k+2}}^\delta
\end{equation*}
and we conclude by computing explicitly $\theta+\theta_1+\theta_2+\theta_3$.
\end{proof}

\noindent {\em Proof of Theorem \ref{main}} It will follow as a consequence of Propositions \ref{modifen}, \ref{R2k+2}, \ref{S2k+2}.
Let
\begin{equation}\label{iterationbound}\sup_{t\in (-\infty, \infty)} \|u(t,x)\|_{\mathcal H^1}=R\end{equation}
then $R<\infty$ by \eqref{basicham}. By integration on the strip $(0,T)$ of the identity \eqref{tobeint} we get
$$
\frac 12\|\partial_t^k A u(T,x)\|_{L^2}^2 + {\mathcal S}_{2k+2} (u(T,x))
=\frac 12\|\partial_t^k A u(0,x)\|_{L^2}^2 
+{\mathcal S}_{2k+2} (u(0,x))+ \int_0^T {\mathcal R}_{2k+2}(u(\tau ,x))d\tau
$$
where $T$ is the one defined in Propositions \ref{R2k+2} and \ref{S2k+2} with $R$ is defined as above.
Then we get 
\begin{equation}\label{secondterm}
\frac 12\|\partial_t^k A u(T,x)\|_{L^2}^2 -\frac 12\|\partial_t^k A u(0,x)\|_{L^2}^2 
\leq C \Big (\|u(0,x) \|_{{\mathcal H}^{2k+2}}^{\frac{4k}{2k+1}+\delta}+ \|u(0,x)\|_{\mathcal{H}^{2k+2}}^{\frac{8k+1}{4k+2}+\delta} \Big).
\end{equation}
Next notice that if we assume that $u(t,x)$ is the nontrivial solution (otherwise the conclusion is trivial) then
$\|u(t,x)\|_{\mathcal{H}^{2k+2}}\geq \|u(t,x)\|_{{L}^2}=const>0$ and the second term on the r.h.s. in \eqref{secondterm} may be absorbed by the first one provided that we modify the multiplicative constant; then
\begin{equation}\label{secondtermprime}
\frac 12\|\partial_t^k A u(T,x)\|_{L^2}^2 -\frac 12\|\partial_t^k A u(0,x)\|_{L^2}^2 
\leq C \|u(0,x)\|_{\mathcal{H}^{2k+2}}^{\frac{8k+1}{4k+2}+\delta}.
\end{equation}
One easily checks that by \eqref{iterationbound} the bound above can be iterated with the same constants, namely
\begin{equation}\label{secondtermprime}
\frac 12\|\partial_t^k A u((n+1)T,x)\|_{L^2}^2 -\frac 12\|\partial_t^k A u(nT,x)\|_{L^2}^2 
\leq C \|u(nT,x)\|_{\mathcal{H}^{2k+2}}^{\frac{8k+1}{4k+2}+\delta}
\end{equation}
for every $n\in \N$. By summing up for $n\in [0, N-1]$ we obtain
$$\|\partial_t^k A u(NT,x)\|_{L^2}^2\leq C \sum_{n\in \{0,\dots,N-1\}} \|u(nT,x)\|_{\mathcal{H}^{2k+2}}^{\frac{8k+1}{4k+2}+\delta}$$
and it implies
$$ \sup_{n\in [0,N]} \|\partial_t^k A u(nT,x)\|_{L^2}^2\leq C N \Big( \sup_{n\in [0, N]} \|u(nT,x)\|_{\mathcal{H}^{2k+2}}\Big)^{\frac{8k+1}{4k+2}+\delta}$$
which implies by \eqref{eqENimproved2d} the estimate
$$\sup_{n\in [0,N]} \|u(nT,x)\|_{\mathcal{H}^{2k+2}}\leq CN^\frac{2(2k+1)}{3}$$
and $\forall N\in \N$,
$$\|u(NT,x)\|_{\mathcal{H}^{2k+2}}\leq CN^\frac{2(2k+1)}{3}\,.$$
By using \eqref{szero} it is easy to deduce that the estimate above implies 
$$\sup_{t\in [NT, (N+1)T]} \|u(t,x)\|_{\mathcal{H}^{2k+2}}\leq CN^\frac{2(2k+1)}{3}$$
provided that we suitably modify the multiplicative constant $C$.
Summarizing we get that, for all $t>0$,
$$\|u(t,x)\|_{\mathcal{H}^{2k+2}}\leq Ct^\frac{2(2k+1)}{3}\,.$$
The same argument works for $t<0$.

\section*{Appendix}
\label{sec:appendix}
We intend to provide a direct proof, based on integration by parts, of the crucial bilinear estimate from \cite{Po},
for solutions to
\begin{equation}
  \label{eq:1}
  i\partial_{t} u -\Delta u +|x|^{2} u=0\,.
\end{equation}
\begin{thm}
Let $1\leq M\leq N$ be dyadic numbers ; for $T\in (0, \infty)$ there exists $C_{T}$ such that
\begin{equation}\label{bilinharmosc}\|u_N v_M\|_{L^2((0,T);L^2)}^2\leq C_T MN^{-1} \|u_N(0)\|_{L^2}^2 \|v_N(0)\|_{L^2}^2\end{equation}
where $u_N$ and $v_N$ are spectrally localized solutions to \eqref{eq:1} (namely
$\Delta_N u_N=u_N$, $\Delta_M v_M=v_M$) respectively with initial datum $u_N(0)$, $v_M(0)$.
\end{thm}
Such bilinear estimates were first obtained for solutions to the classical linear Schr\"odinger equation in \cite{b2}, using direct computations in Fourier variables. In \cite{CKSTT}, so-called interaction Morawetz interaction estimates were introduced for the 3D nonlinear Schr\"odinger equation, relying on a bilinear version of the classical Morawetz estimate. Here, we rely on the bilinear computation from \cite{PlVe} that non only extended such bilinear virial estimates to low dimensions but also allowed to recover Bourgain's estimates from \cite{b2}. We will follow the strategy from \cite{FP} where bilinear estimates on bounded domains were obtained, bypassing the need for Fourier localization. We split the proof in several steps. First we prove that, for a given $T\in (0, \infty)$:
\begin{equation}\label{eq:6new}
    \int_{0}^{T} \Big(\int \int_{|x-y|<\frac 1 M} M  |u_{N}(x) \nabla_{y} \bar v_{M}(y)+\bar v_{M}(y)\nabla_{x}u_{N}(x) |^{2} \,dxdy\Big)
    dt \leq  C_{T} N\|u_{N}(0)\|^{2}_{L^{2}}\|v_{M}(0)\|^{2}_{L^{2}}.
\end{equation}
Next we deduce from \eqref{eq:6new} that
\begin{equation}
  \label{eq:555new}
  \int_{0}^{T} \Big (\int |\nabla_x (v_{M}  u_{N})|^{2}\, dx\Big) dt \leq C_{T} M N\|u_{N}(0)\|^{2}_{L^{2}} \|v_{M}(0)\|^{2}_{L^{2}}.
\end{equation}
Estimate \eqref{eq:555new}, along with a companion easier estimate for
$\int_{0}^{T} \Big( \int |x|^2 | v_{M}  u_{N}|^{2}\, dx \Big)\,dt$, implies
\begin{equation} \label{eq:555newnewnew}
  \int_{0}^{T} \|v_{M} \bar u_{N}\|^{2}_{\mathcal{H}^{1}}\,dt \leq C_{T} M N\|u_{N}(0)\|^{2}_{L^{2}} \|v_{M}(0)\|^{2}_{L^{2}}\,.
\end{equation}
Finally, by a spectral localization argument, we prove that \eqref{eq:555newnewnew} implies \eqref{bilinharmosc}.
\subsection*{Proof of (\ref{eq:6new})} We first remark for later use
that once \eqref{eq:6new} will be established, then we are allowed to replace 
$v_N$ by $Av_N$ (which is still a localized solution to \eqref{eq:1}) and we get:
 \begin{multline}\label{eq:6newold}
    \int_{0}^{T} \Big(\int \int_{|x-y|<\frac 1 M} M |u_{N}(x) \nabla_{y} (A \bar v_{M})(y)+ (A\bar v_{M})(y)\nabla_{x}u_{N}(x)|^{2} \,dxdy\Big) 
    dt\\\leq  C_{T} NM^2\|u_{N}(0)\|^{2}_{L^{2}}\|v_{M}(0)\|^{2}_{L^{2}}.
\end{multline}
Next we focus on the proof of \eqref{eq:6new} and from now on, $T$ is fixed in $(0,+\infty)$. Let  $\rho:\R^2\rightarrow \R$ be a $C^{1}$ function whose derivative is piecewise differentiable,  with $H_{\rho}$ denoting the bilinear form associated to its Hessian (as a distribution), $H_{\rho}(a,b)=\sum_{k,l}(\partial^{2}_{kl}\rho) a_{k}b_{l}$; all $\partial^{2}_{kl}\rho$ are actually piecewise continuous functions and under such assumptions all subsequent integrations by parts are fully justified in the classical sense. We claim that, for any couple of solutions $u,v$ of \eqref{eq:1},
\begin{multline}\label{intervrai}
\int_0^T \Big( \int \int H_{\rho}(x-y)(\bar v(y)\nabla_x u(x)+u(x)\nabla_y \bar v(y), v(y)\nabla_x \bar u(x)+\bar u(x)\nabla_y v(y) )\, dxdy \Big)
dt
\\\leq C_T \|\nabla \rho\|_{L^\infty} (\|v(0)\|_{L^2}^2 \|u(0)\|_{L^2} \|u(0)\|_{{\mathcal H}^1}+
\|u(0)\|_{L^2}^2 \|v(0)\|_{L^2} \|v(0)\|_{{\mathcal H}^1})\end{multline}
where we dropped time dependence for notational simplicity. Following  \cite{FP} we define a convex function $\rho_M:\R\rightarrow \R$,
$$
\rho_M(z)= \begin{cases}\frac M2 z^2+\frac 1{2M}, \quad z\leq \frac 1M\\
z, \quad z>\frac 1M\end{cases}
$$
and we use \eqref{intervrai} with $\rho(x-y)=\rho_M(x_1-y_1)$: we get, by direct computation of the Hessian $H_{\rho}$,
\begin{multline}\label{intervrai1d}
\int_0^T \Big (\int \int_{|x_1-y_1|<\frac 1M} M |(\bar v(y)\partial_{x_1} u(x)+u(x)\partial_{y_1} \bar v(y)|^2\, dxdy \Big )
dt\\\leq C_T( \|v(0)\|_{L^2}^2 \|u(0)\|_{L^2} \|u(0)\|_{{\mathcal H}^1}+ \|u(0)\|_{L^2}^2 \|v(0)\|_{L^2} \|v(0)\|_{{\mathcal H}^1})\end{multline}
where there is no contribution in the region $|x_1-y_1|>\frac 1M$ as $H_{\rho}=0$ there, and we used that $\|\rho_M'(z)\|_{L^\infty}\leq 1$. 
Of course by choosing $\rho(x-y)=\rho(x_2-y_2)$ we  get a similar estimate where $x_1, y_1$ are replaced by $x_2, y_2$ and by combining the two estimates we get 
\eqref{eq:6new} where we noticed that $|x-y|<\frac 1M\subset \max\{|x_1-y_1|, |x_2-y_2|\}<\frac 1M$.
Replacing $u$ and $v$ by $u_N$ and $v_M$ and using spectral localization we get 
\eqref{eq:6new}.

Next we focus on the proof of \eqref{intervrai}. We compute the second derivative w.r.t. time of the functional
$$I_\rho(t)=\int \int |u(x)|^2 \rho(x-y) |v(y)|^2 \, dx dy$$
where for for simplicity we have dropped the time dependence of $u, v$. In order to do so, recall that by the classical virial computation
we get for a solution $w(t,x)$ to \eqref{eq:1} (we drop again time-dependence of $w$ and set $\rho_{y}(x)=\rho(x-y)$ to emphasize that $y$ is a fixed base point here):
\begin{gather}\label{1vir}\frac d{dt} \int \rho_{y}(x) |w(x)|^2\, dx= 2\int \nabla \rho_{y}(x) \cdot \Im (\nabla \bar w(x) w(x))\, dx\\
\label{2vir}\;\;\;\;\;\;\;\;\frac {d^2}{dt^2} \int \rho_{y}(x) |w|^2\, dx= 4\int H_{\rho_{y}} (\nabla w(x), \nabla \bar w(x)) - \int \Delta \rho_{y}(x) \Delta(|w(x)|^2)\, dx -
4 \int x\cdot \nabla \rho_{y}(x)  |w(x)|^2 dx\,,
\end{gather}
where we emphasize that we will not be using more than two derivatives on $\rho_{y}$. Next, using \eqref{1vir} we get
\begin{equation}
  \frac d{dt} I_\rho(t)
  = 2\int \int \nabla \rho(x-y)\cdot \Im (\nabla_x \bar u(x) u(x)) |v(y)|^2\, dx dy
  -2 \int \int \nabla \rho(x-y)\cdot \Im (\nabla_y \bar v(y) v(y)) |u(x)|^2\, dx dy\,.
  \end{equation}
Using that $\|\nabla w\|_{L^2}\leq C \| w\|_{{\mathcal H}^1}$ at fixed time followed by conservation of mass and energy for \eqref{eq:1}, we get
\begin{align}\label{boiundarter}|\frac d{dt} I_\rho(t)|\leq & 2
\|\nabla \rho\|_{L^\infty} (\|v\|_{L^2}^2 \|u\|_{L^2} \|\nabla u\|_{L^2} + \|u\|_{L^2}^2 \|v\|_{L^2} \|\nabla v\|_{L^2})
  \\
  \leq & C \|\nabla \rho\|_{L^\infty} (\|v(0)\|_{L^2}^2 \|u(0)\|_{L^2} \|u(0)\|_{{\mathcal H}^1}
  +\|u(0)\|_{L^2}^2 \|v(0)\|_{L^2} \|v(0)\|_{{\mathcal H}^1} )\,.
  \end{align}
 For later use notice also that using \eqref{1vir} on both mass densities, 
 \begin{align}\label{twodt}
 \int \Big (\int \rho(x-y) \frac d{dt} |u(x)|^2\, dx \Big ) \frac d{dt}|v(y)|^2 \, dy= & 2 \int
  \Big (\int \nabla \rho(x-y) \cdot \Im(\nabla_x \bar u(x) u(x))\, dx\Big ) \frac d{dt} |v(y)|^2 \, dy\\
 = & -4 \int\int H_{\rho}( \Im (\nabla_x \bar u(x) u(x)), \Im (\nabla_y \bar v(y) v(y)) \, dx dy\,.
 \end{align}
 On the other hand by combining %
 \eqref{2vir} and \eqref{twodt},
 \begin{align}
 \frac {d^2}{dt^2} I_\rho(t) %
 = & 4\int\int  H_{\rho}(x-y) (\nabla u(x) ,\nabla \bar u(x)) |v(y)|^2 +4\int \int H_{\rho}(x-y) (\nabla v(y) ,\nabla \bar v(y))|u(x)|^2\, dxdy  \\
 & {} - \int \int \Delta\rho(x-y)\Delta (|u(x)|^2) |v(y)|^2 \, dxdy - \int\int  \Delta \rho(x-y) |u(x)|^2\Delta (|u(y)|^2)\, dxdy \\
 & {} -8 (\int H_{\rho}(x-y) (\Im (\nabla \bar u(x) u(x)),   \Im (\nabla \bar v(y) v(y))) \, dy dx\\
 &{}  -4\Re \int \int (x-y)\cdot  \nabla \rho(x-y) |u(x)|^2 |v(y)|^2 \,dx dy=I+II+III+IV+V+VI\,.
\end{align}
Following \cite{PlVe}, we rewrite $\Delta\rho(x-y)=
-\nabla_x \cdot  \nabla_y \rho(x-y)$ and integrate  by parts w.r.t. to $x$ and $y$:
$$III+IV=8 \int \int H_{\rho}(x-y) (\Re (\bar u(x) \nabla u(x)) , \Re (\bar v(y) \nabla v(y)))\, dxdy\,.$$
Now, thinking about just one direction of derivation, we have the following identity
\begin{multline*}
 4|v|^2(y) |\partial u|^2(x)+ 4|u|^2(x) |\partial v|^2(y)+ 8\frac{(v\partial\bar v+\bar v \partial v)(y)}2 \frac{(u\partial\bar
       u+\bar u \partial u)(x)} 2- 8\frac{(v\partial\bar v-\bar v \partial v)(y)}2\frac{ (\bar
       u\partial u-  u \partial\bar u)(x)}2\\
 =  4|v|^2(y) |\partial u|^2(x)+ 4|u|^2(x) |\partial
 v|^2(y)+ 4 v\partial\bar v(y) u\partial\bar u(x) +4 \bar v \partial
 v(y)\bar u\partial u(x) =   4 |\bar v (y)\partial u(x)+u(x)\partial \bar v(y)|^2\,,
\end{multline*}
which allows to recombine $I+II+III+IV+V$, to get
\begin{multline*}
\frac {d^2}{dt^2} I_\rho(t)= 4\int \int H_{\rho}(x-y)(\bar v(y)\nabla u(x)+u(x)\nabla \bar v(y) ,v(y)\nabla \bar u(x)+\bar u(x)\nabla v(y) )\, dxdy \\
-4\Re \int\int  \nabla \rho(x-y) \cdot (x-y) |v(y)|^2|u(x)|^2\, dxdy.\end{multline*}
After integration in time of the identity above and by recalling \eqref{boiundarter} we get
\begin{multline*}
4\int_0^T \Big(\int \int H_{\rho}(x-y)(\bar v(y)\nabla_x u(x)+u(x)\nabla_y \bar v(y) ,v(y)\nabla_x \bar u(x)+\bar u(x)\nabla_y v(y) ) \, dxdy \Big)
dt\\\leq C \|\nabla \rho\|_{L^\infty}  (\|v(0)\|_{L^2}^2 \|u(0)\|_{L^2} \|u(0)\|_{{\mathcal H}^1}
+\|u(0)\|_{L^2}^2 \|v(0)\|_{L^2} \|v(0)\|_{{\mathcal H}^1} ) 
\\+ 4\int_0^T \Big(\int\int  |\nabla \rho(x-y)| |y-x| |v(y)|^2|u(x)|^2\, dxdy\Big)dt\\
\leq C \|\nabla \rho\|_{L^\infty}  (\|v(0)\|_{L^2}^2 \|u(0)\|_{L^2} \|u(0)\|_{{\mathcal H}^1}
+\|u(0)\|_{L^2}^2 \|v(0)\|_{L^2} \|v(0)\|_{{\mathcal H}^1} ) 
\\+ C_T  \|\nabla \rho\|_{L^\infty} \sup_{t\in (0,T)}
(\|u(t)\|_{L^2}^2 \|v(t)\|_{L^{2}}\| y v(t)\|_{L^2} +\|v(t)\|_{L^2}^2 \|u(t)\|_{L^{2}}\| x u(t)\|_{L^2}^2 )\,.
\end{multline*}
Using  $\|y w\|_{L^2}^2\leq \|w\|_{{\mathcal H}^1}$ and, again, conservation of mass and  energy for \eqref{eq:1}, this estimate implies \eqref{intervrai}.

\subsection*{Proof of (\ref{eq:6new}) $\Rightarrow$ (\ref{eq:555new})}
We need a suitable local elliptic estimate for our operator $A=-\Delta+|x|^{2}$ to reproduce the computation from \cite{FP}.
The next lemma is a modification of Lemma 4.2 in \cite{FP}.
\begin{lem}\label{ellip}
There exists $C>0$ and $\lambda_{0}\geq 1$ such that, for  any smooth function $\phi$ in $\R^2$ and $\lambda\geq \lambda_{0}$, 
the following pointwise estimate holds:
  \begin{equation}
    \label{eq:tracelem}
    |\phi(x)|^2\leq C \lambda^{-2} \int_{|x-y|<{\lambda^{-1}}}
    |A \phi|^2 \,dy +C \lambda^2 \int_{|x-y|<{\lambda^{-1}}} |\phi|^2\,dy, \quad \forall x\in \R^2\,.
  \end{equation}
\end{lem}
\begin{proof}
Without loss of generality we may restrict to real-valued $\phi$. The lemma is proved in \cite{FP} if we replace in the r.h.s. the operator $A$ by $-\Delta$ 
and  the domain of integration by the smaller domain $|x-y|<(4\lambda)^{-1}$ (this fact follows from classical elliptic theory and Sobolev embedding for $\lambda=1$ and then any $\lambda>0$ by rescaling). Thus we conclude provided that we prove
\begin{equation}\label{aim}
\lambda^{-2} \int_{|x-y|<{(4\lambda)^{-1}}}
    |\Delta \phi|^2 \,dy\leq C \lambda^{-2} \int_{|x-y|<{\lambda^{-1}}}
    |A \phi|^2 \,dy+C \lambda^{2} \int_{|x-y|<{\lambda^{-1}}}
    |\phi|^2 \,dy.\end{equation}
In order to prove this estimate we expand the square $\int |\Delta f|^{2} =\int |Af-|y|^{2}f |^{2}$ and after integrations by parts we get
  \begin{equation}
    \label{eq:H2norm}
      \int (|\Delta f|^{2}+ |y|^{4} |f|^{2}+ 2 |y|^{2}|\nabla f|^{2})\,dy = \int (|Af|^{2}+4 |f|^{2})\,dy
  \end{equation}
for any real-valued function $f\in C^\infty_0(\R^2)$.
Next we pick $f(y)=\chi_{\lambda}(y) \phi(y)$, where $\chi_{\lambda}(y)=\chi( \lambda(y-x))$, with
$\chi(|z|)=1$ on $|z|<\frac 14$ and $\chi(|z|)=0$ on $|z|>\frac 12$. All subsequent cutoffs w.r.t. $y$ will be centered at $x$. Expanding $\int |\Delta(\chi_{\lambda} \phi)|^{2}$ and $\int |A(\chi_{\lambda} \phi)|^{2}$ and replacing in the previous identity we get:
\begin{multline}\label{impsa}
  \int (|\chi_{\lambda}|^2 |\Delta \phi|^{2}+ |y|^{4} |\chi_{\lambda}|^2 |\phi|^{2}+2 |y|^{2}|\nabla(\chi_{\lambda} \phi)|^{2})\, dy \\
  = \int (|\chi_{\lambda}|^2|Av|^{2}+4|\chi_{\lambda}|^2| \phi|^{2}) \, dy- 2 \int \chi_{\lambda} |y|^{2} \phi (2\nabla \chi_{\lambda}\cdot \nabla \phi+ \Delta \chi_{\lambda} \phi )\, dy\,.
 \end{multline}
  By Cauchy-Schwarz and elementary manipulations we estimate the last term on the r.h.s. as follows, for every $\mu>0$ and with a universal constant $C>0$:
  $$\Big |\int \chi_{\lambda} |y|^{2}  \phi (2\nabla \chi_{\lambda}\cdot  \nabla \phi+\phi  \Delta \chi_{\lambda})\, dy\Big |^2 \leq C\mu  \int |y|^4 |\chi_{\lambda}|^2 |\phi|^2\, dy
 +\frac C\mu \int (|\nabla \chi_{\lambda} \cdot \nabla \phi|^{2}+ |\Delta \chi_{\lambda}|^{2} |\phi|^{2})\, dy.$$
If we choose the constant $\mu$ small enough then we can absorb $\int |y|^4 |\chi_{\lambda} v|^2 dx$ on the l.h.s. in \eqref{impsa}
and by neglecting some positive terms we get, abusing notation for the constant $C$,
\begin{equation}\label{impsash}
  \int |\chi_{\lambda}|^2| \Delta \phi|^{2}\,dy \leq C \int (|\chi_{\lambda}|^2| A\phi|^{2}+4|\chi_{\lambda}|^2| \phi|^{2} + |\nabla \chi_{\lambda} \cdot \nabla \phi|^{2}+ |\Delta \chi_{\lambda}|^{2})
 |\phi|^{2} \, dy
 \end{equation}
 and by elementary considerations
 \begin{equation*}
  \int_{|x-y|<(4\lambda)^{-1}} |\Delta \phi|^{2}\, dy  \leq C \int_{|x-y|<(2\lambda)^{-1}} |A\phi |^{2}\,dy+ C \lambda^{2} \int \tilde \chi_{\lambda}  |\nabla \phi |^{2}\,dy +
 C(1+\lambda^2) \int_{|x-y|<(2\lambda)^{-1}} |\phi|^{2}\, dy
 \end{equation*}
  where $\tilde \chi_{\lambda}$ is a suitable enlargement of $\chi_{\lambda}$, namely
  $\tilde \chi_\lambda(y)=\tilde \chi \big( \frac {y-x} \lambda\big)$, with
$\tilde \chi(|z|)=1$ on $|z|<\frac 12$ and $\tilde \chi(|z|)=0$ on $|z|>1$. Then \eqref{aim} follows provided that 
$$\int \tilde \chi_{\lambda} |\nabla \phi|^{2}dy\leq C \lambda^{-2} \int_{|x-y|<{\lambda^{-1}}}
    |A \phi|^2 \,dy + C\lambda^2 \int_{|x-y|<{\lambda^{-1}}} |\phi|^2\,dy.$$
 In order to do that we write (either integrating by parts or replacing $-\Delta$ by $A-|x|^2$)
  \[
    -2  \int \tilde \chi_{\lambda}  \phi \Delta \phi\, dy=2\int \tilde \chi_{\lambda} |\nabla  \phi |^{2}\, dy- \int \Delta \tilde \chi_{\lambda}  |\phi|^{2}
    \, dy=2  \int \tilde \chi_{\lambda} \phi A \phi  \, dy-2\int |y|^{2} \tilde \chi_{\lambda} |\phi|^{2}\, dy
  \]
  and hence
  \begin{multline*}
   2 \int \tilde \chi_{\lambda} |\nabla  \phi |^{2}\, dy+ 2\int |y|^{2} \tilde \chi_{\lambda} |\phi |^{2} \, dy =-2   \int \tilde \chi_{\lambda}  \phi A \phi
    \, dy+ \int \Delta \tilde \chi_{\lambda}  |\phi|^{2} \, dy\\
     \leq  C \lambda^{-2}\int_{|x-y|<\lambda^{-1}} |A \phi|^{2} \, dy+ C (1+\lambda^{2})\int_{|x-y|<\lambda^{-1}} |\phi|^{2}\, dy
  \end{multline*}
  where we used Cauchy-Schwarz at the last step.
  \end{proof}

We now proceed to prove that \eqref{eq:6new} $\Rightarrow$ \eqref{eq:555new}. The first term in the square at the l.h.s. of \eqref{eq:6new} turns out to be lower order:  we compute
by change of variable, Cauchy-Schwarz inequality and  Strichartz estimate,
\begin{multline}
  \label{eq:8new}
   \int_{0}^{T} \Big(\int \int_{|x-y|<\frac 1 M} |u_{N}(x) \nabla_{y} \bar v_{M}(y)|^2 dxdy \Big ) dt=
    \int_{0}^{T} \int_{|z|<\frac 1 M} |u_{N}(x) \nabla_{x} \bar v_{M}(x-z)|^{2} \,dxdzdt 
    \\
    \leq \int_{|z|<\frac 1 M} \| u_{N}\|_{L^{4}((0,T);L^4)}^{2}\|\nabla v_{M}\|_{L^{4}
    ((0,T);L^4)}^{2}\, dz
    \leq C_{T}  \|u_{N}(0)\|^{2}_{2}  \|v_{M}(0)\|^{2}_{L^{2}} 
\end{multline}
where at the last step we used that $\|\nabla v_{M}\|_{L^{4}
    ((0,T);L^4)}\leq C_TM\|v_{M}(0)\|_{L^{2}}$. In turn this bound follows by noticing that 
    $\nabla v_{M}$ is solution to the inhomogeneous equation associated with \eqref{eq:1} with forcing term $2x v_M$. Hence
    by the inhomogeneous Strichartz estimate,  placing the forcing term in $L^1((0,T);L^2)$,
\begin{equation}
  \label{eq:9new}
    \| \nabla v_{M}\|_{L^{4}((0,T);L^{4})} \leq C \| \nabla v_{M}(0)\|_{L^{2}}+C\||x|v_{M}\|_{L^1((0,T);L^{2})}\leq C_T \| v_{M}(0)\|_{\mathcal{H}^{1}}
    \leq C_T M \| v_{M}(0)\|_{L^2} \,,
\end{equation}
where we used conservation of energy for \eqref{eq:1} and the bound
$\||x| w\|_{L^2}\leq C \|w\|_{{\mathcal H}^1}$ for every time independent function.\\
Recall that \eqref{eq:8new} holds with $v_{M}$ replaced by $Av_{M}$ (it is still a solution to \eqref{eq:1}), hence 
 we get:
  \begin{equation}
    \label{eq:7newnew}
    \int_{0}^{T}\Big( \int \int_{|x-y|<\frac 1 M}  |u_{N}(x) \nabla_y(A \bar v_{M})(y)|^{2} \,dxdy\Big)dt \leq C_{T} 
    M^4 \|u_{N}(0)\|^{2}_{L^{2}} \|v_{M}(0)\|^{2}_{L^{2}}\,.
  \end{equation}
   We now proceed using the Lemma \ref{ellip} and we get
  \begin{multline*}
         \int_{0}^{T} \Big(\int |\bar v_{M}(x)\nabla_{x}u_{N}(x)|^{2} \,dx\Big) dt 
         \leq   C  \int_{0}^{T} \Big( \int \int_{|x-y|<\frac 1 M}  M^{2}|\bar v_{M}(y)\nabla_{x}u_{N}(x)|^{2} \\
         {}+\frac 1 {M^{2}} |A\bar v_{M}(y)\nabla_{x}u_{N}(x)|^{2} \,dxdy
         \Big) dt
  \end{multline*} that by  \eqref{eq:6new} and \eqref{eq:6newold} implies
  \begin{multline*}
  (\dots)\leq C\int_{0}^{T} \Big( \int \int_{|x-y|<\frac 1 M}  M^{2}|u_{N}(x) \nabla_{y} \bar v_{M}(y)|^{2}+\frac 1 {M^{2}} |u_{N}(x)\nabla_{y} (A\bar v_{M})(y)|^{2} \,dxdy
         \Big) dt 
 \\ +C_{T} NM\|u_{N}(0)\|^{2}_{L^{2}}\|v_{M}(0)\|^{2}_{L^{2}}\,.
  \end{multline*}
Combining the above estimate with \eqref{eq:8new} and \eqref{eq:7newnew} we obtain
\begin{multline}\label{lensa}
\int_{0}^{T} \Big(\int |\bar v_{M}(x)\nabla_{x}u_{N}(x)|^{2} \,dx\Big) dt \leq C_{T} (M^2 +NM)\|u_{N}(0)\|^{2}_{L^{2}}\|v_{M}(0)\|^{2}_{L^{2}}
\\\leq C_T NM \|u_{N}(0)\|^{2}_{L^{2}}\|v_{M}(0)\|^{2}_{L^{2}}\,.
\end{multline}
On the other hand by Cauchy-Schwarz, Strichartz estimate and \eqref{eq:9new} we have
\begin{equation*}
  \int_{0}^{T} \Big(\int |u_{N}(x)\nabla_x \bar v_{M}(x)|^2 \, dx\Big)dt \leq 
   \| u_{N}\|_{L^{4}((0,T);L^4)}^{2}\|\nabla v_{M}\|_{L^{4}
    ((0,T);L^4)}^{2} \leq C_{T} M^{2}  \|v_{M}(0)\|_{L^{2}}^{2} \|u_{N}(0)\|^{2}_{L^{2}}\,.
\end{equation*}
Therefore combining this last estimate with  \eqref{lensa}  we get  \eqref{eq:555new}.
\subsection*{Proof of (\ref{eq:555newnewnew})} Due to \eqref{eq:555new}, it suffices to prove
\begin{equation}\label{bertm}
  \int_{0}^{T} \Big( \int  |x|^2 |v_{M} \bar u_{N}|^{2} dx\Big) dt \leq C_{T} M N\|u_{N}(0)\|^{2}_{L^{2}} \|v_{M}(0)\|^{2}_{L^{2}}.
\end{equation}
By H\"older inequality we have

\begin{equation}\label{bertmaxi}
  \int_{0}^{T}\Big ( \int  |x|^{2} |v_{M} \bar u_{N}|^{2}\,dx \Big)dt 
  \leq \||x| v_M\|^2_{L^4((0, T);L^4)} \|u_N\|^2_{L^4((0, T);L^4)}\,.
\end{equation}
Next notice that  $|x|^{2}v_{M}$ is solution to the inhomogeneous equation associated with \eqref{eq:1}, with source term $-4v_{M}-2x\cdot \nabla v_{M}$. Again, using Strichartz and placing the source term in $L^1((0,T); L^2)$,
\begin{align}
  \label{eq:99}
    \| |x|^{2} v_{M}\|_{L^{4}((0,T);L^{4})} \leq & C  \| |x|^{2} v_{M}(0)\|_{L^{2}}+C \|v_{M}\|_{L^1((0,T);L^{2})}+C\| x\cdot \nabla v_{M}
                                                   \|_{L^1((0,T);L^{2})} \\
   \leq & C_T \| v_{M}(0)\|_{\mathcal{H}^{2}}\leq C_T M^2 \| v_{M}(0)\|_{L^{2}}
\end{align}
where we used the time independent estimate 
$\| x\cdot \nabla w\|_{L^2}\leq C \|w\|_{{\mathcal H}^2}$ (see \eqref{eq:H2norm}) and conservation of the ${\mathcal H}^2$ norm for \eqref{eq:1}.
Interpolation between \eqref{eq:99} and the Strichartz estimate
$\|v_{M}\|_{L^{4}((0,T);L^{4})}\leq C \|v_M(0)\|_{L^2}$ implies
$\| |x| v_{M}\|_{L^{4}((0,T);L^{4})} \leq C M \|v_M(0)\|_{L^2}$. Combining this estimate, Strichartz for $u_{N}$ and \eqref{bertmaxi} we obtain \eqref{bertm} (actually, a stronger version of \eqref{bertm} as on the r.h.s. we  get $M^2$).
\subsection*{Proof of the implication (\ref{eq:555newnewnew}) $\Rightarrow$ (\ref{bilinharmosc})}
We can write
\begin{equation}
  \label{eq:10}
  \| v_{M}  u_{N}\|^{2}_{L^{2}((0,T);L^2)}=\sum_{K\in 2^{\mathbb{N}}} \|\Delta_{K}( v_{M}u_{N})\|^{2}_{L^2((0,T);L^2)}
\end{equation}
If $K>N$, we may forget about $\Delta_{K}$ and use \eqref{eq:555newnewnew} in order to get
\begin{equation}\label{derderder}
\sum_{K>N} \|\Delta_{K}(v_{M} u_{N})\|^{2}_{L^2((0,T);L^2)}
\leq C \sum_{K>N} (1+K)^{-2}
 \|  v_{M}u_{N}\|_{L^2((0,T);{\mathcal H}^1)}^{2} 
 \leq C_{T} M N^{-1}\|u_{N}(0)\|^{2}_{L^{2}} \|v_{M}(0)\|^{2}_{L^{2}}  
\end{equation}
For $K\leq N$, denote $S_{N}=\sum_{K\leq N} \Delta_{K}$ and write directly
\begin{equation}\label{tilde}
  S_{N}( v_{M} u_{N})=S_{N} ( v_{M} N^{-2} A \tilde u_{N})
\end{equation}
where $\tilde u_{N}=\tilde \Delta_{N} u_{N}$ and the localization operator $\tilde \Delta_{N}$ was chosen so that $N^{-2} A \tilde \Delta_{N}$ is the identity on the support of $\Delta_{N}$. We may now write $ v_{M} A \tilde u_{N}= A ( v_{M} \tilde u_{N})+ \tilde u_{N} \Delta  v_{M}+2 \nabla  v_{M} \cdot \nabla \tilde u_{N}$ and hence
by \eqref{tilde}, uniform boundedness of $S_N$  and $N^{-1}\sqrt A S_{N}$ on $L^{2}$, we get
\begin{align*}
  \| S_{N}( v_{M} u_{N})\|_{L^{2}}^{2} \leq &  CN^{-2} \|N^{-1} \sqrt A S_{N} (\sqrt A( v_{M}\tilde u_{N})\|_{L^{2}}^{2}+ CN^{-4} (\|  \tilde u_{N} \Delta v_{M}\|^{2}_{L^{2}}+
  \|\nabla  v_{M} \cdot \nabla \tilde u_{N}\|^{2}_{L^{2}})\\
  \leq &  C N^{-2} \| \sqrt A( v_{M}\tilde u_{N})\|_{L^{2}}^{2}+ C N^{-4}\| \Delta  v_{M}\|^{2}_{L^{4}} \|\tilde u_{N}\|^{2}_{L^{4}}+ C N^{-4} \|\nabla  v_{M} \|^{2}_{L^{4}}\|\nabla \tilde u_{N}\|^{2}_{L^{4}}\,.
  \end{align*} 
After integration in time, using Strichartz estimates to control $L^4$ norms (use \eqref{eq:9new}
  to control $\|\nabla v_M\|_{L^{4}_{t,x}}$ and a similar argument to control $\|\nabla \tilde u_N\|_{L^{4}_{t,x}}$) and \eqref{eq:555newnewnew}, we get

  \begin{align*}
    \int_{0}^{T}  \| S_{N}( v_{M} u_{N})\|_{L^{2}}^{2} \,dt  \leq & C N^{-2} \int_{0}^{T}  \|  v_{M}\tilde u_{N}\|_{\mathcal{H}^{1}}^{2}\,dt + C_T M^2N^{-4}(M^{2}+N^{2})\|  v_{M}(0)\|^{2}_{L^{2}} \|\tilde u_{N}(0)\|^{2}_{L^{2}} \\
    \leq &  C_T  MN^{-1} \| v_{M}(0)\|^{2}_{L^{2}}\|\tilde u_{N}(0)\|^{2}_{L^{2}}+ C_{T} M^{2}N^{-2}\| v_{M}(0)\|^{2}_{L^{2}} \|\tilde u_{N}(0)\|^{2}_{L^{2}}\,,
\end{align*}
and we complete the proof with $\|\tilde u_{N}(0)\|_{L^{2}}\leq C\|u_{N}(0)\|_{L^{2}}$.\qed

\end{document}